\documentclass[12pt]{amsart}  

\usepackage{custom_preamble_arXiv}

\begin{document}

\title{Solving $xz=y^2$ in certain subsets of finite groups}

\author{\tsname}
\address{\tsaddress}
\email{\tsemail}

\begin{abstract}
Suppose that $G$ is a finite group and $A \subset G$ has no non-trivial solutions to $xz=y^2$.  We shall show that $|A|=|G|/(\log \log |G|)^{\Omega(1)}$.
\end{abstract}

\maketitle

\section{Introduction}\label{sec.introduction}

Suppose that $G$ is a group and $A \subset G$.  We say that $A$ contains a solution to $xz=y^2$ if there is a triple $(x,y,z) \in A^3$ such that $xz=y^2$; we say the solution is non-trivial if $x \neq y$.  This paper is concerned with the following result of Bergelson, McCutcheon and Zhang.
\begin{theorem}[{\cite[Corollary 6.5]{bermcczha::}}]\label{thm.ksv}
Suppose that $G$ is a finite group and $A \subset G$ contains no non-trivial solutions to $xz=y^2$.  Then $|A|=o(|G|)$.
\end{theorem}
Note that if $G$ is Abelian then $xz=y^2$ and $x \neq y$ if and only if $(x,y,z)=(x,x+d,x+2d)$ for some $d\neq 0$ \emph{i.e.} if any only if $(x,y,z)$ is a non-trivial three-term arithmetic progression.  With this observation in hand it is a short argument to get Roth's celebrated theorem on three-term arithmetic progressions \cite{rot::,rot::0} from Theorem \ref{thm.ksv} in the special case when $G$ is a cyclic group.  This should should give some idea of why Theorem \ref{thm.ksv} might be of interest.

Theorem \ref{thm.ksv} was proved in the case when $G$ is Abelian (and of odd order) in \cite[Theorem 1]{fragrarod::}, in which paper the authors attribute the result to \cite{brobuh::0}.  The odd order condition here is a technical convenience which can be ignored on a first reading.  Frankl, Graham, and R{\"o}dl's proof goes by the triangle removal lemma of Ruzsa and Szemer{\'e}di \cite{ruzsze::}, and it was noted by Kr{\'a}l, Serra, and Vena, in the course of a wider generalisation, that the removal lemma can also be used to prove Theorem \ref{thm.ksv} \cite[Corollary 3]{kraserven::}.  

The removal lemma suffers from notorious poor dependencies and the reader is directed to \cite{confox::} for a discussion of this and related matters.  For our purposes it suffices to know that the best are due to Fox \cite{fox::} and inserting his work into the proof of \cite[Corollary 3]{kraserven::} would give
\begin{equation}\label{eqn.b}
|A|=\frac{|G|}{\exp(\Omega(\log_* |G|))}
\end{equation}
when $A$ is a set satisfying the hypotheses of Theorem \ref{thm.ksv} (and $G$ is a group of odd order).  Here for $R \in \N$ we define $\log_*R$ to be the minimal $n \in \N_0$ such that
\begin{equation*}
\overbrace{\log_2(\log_2 \dots (\log_2}^{n \text{ times}} R )) \leq 1.
\end{equation*}
This function grows more slowly than any finite composition of logarithms.  It is the purpose of this paper to improve on the bound (\ref{eqn.b}) by showing the following.
\begin{theorem}\label{thm.mainintro}
Suppose that $G$ is a finite group and $A \subset G$ contains no non-trivial solutions to $xz=y^2$.  Then
\begin{equation*}
|A|=\frac{|G|}{(\log \log |G|)^{\Omega(1)}}.
\end{equation*}
\end{theorem}
For various classes of Abelian groups there has been considerable work improving the bounds in Theorem \ref{thm.mainintro}.  The best known arguments are due to Bloom \cite{blo::0}, although he concentrates on the case of initial segments of the integers, he uses the framework of Bourgain \cite{bou::5} and so the argument easily extends to finite Abelian groups to show the following.
\begin{theorem}\label{thm.bloom}
Suppose that $G$ is a finite Abelian group and $A \subset G$ contains no non-trivial solutions to $x+z=2y$.  Then
\begin{equation*}
|A|=\frac{|G|}{\log^{1-o(1)}|G|}.
\end{equation*}
\end{theorem}
A small amount of additional care is needed to avoid restricting attention to groups of odd order, but it will be useful to avoid this restriction in the remaining discussion.

It is worth noting that when $G$ is non-Abelian there are two possible notions of three-term arithmetic progression.  The first is triples $(x,y,z)$ such that $xz=y^2$; the second is triples $(x,y,z)$ such that $z=yx^{-1}y$.  This second notion is, perhaps, more natural since it is left (and so by symmetry right) invariant meaning that if $(x,y,z)$ has $z=yx^{-1}y$ then $(ax,ay,az)$ has $az=ay(ax)^{-1}ay$.  The first notion is notion is not, in general, left or right invariant.  (This does not, however, lead to the problems that non-translation invariant equations have in Abelian groups.  For example, if $G=\Z/2N\Z$ then the odd numbers form a subset of $G$ of density $\frac{1}{2}$ not containing any solutions to $x+y=z$.)  Nevertheless our proof is iterative and the lack of translation invariance does present issues. These are discussed in more detail at the start of \S\ref{sec.u1}.

In the same way as before we say that a solution to $z=yx^{-1}y$ is non-trivial if $x \neq y$.  The next result is a slight variant of \cite[Theorem 2.5]{sol::0} also communicated to the author personally by Ernie Croot, and is included for comparison with Theorem \ref{thm.mainintro}.
\begin{theorem}\label{thm.sol}
Suppose that $G$ is a finite group and $A \subset G$ contains no non-trivial solutions to $z=yx^{-1}y$.  Then
\begin{equation*}
|A|=\frac{|G|}{\log^{\frac{1}{2}-o(1)} |G|}.
\end{equation*}
\end{theorem}
\begin{proof}
By a result of Pyber \cite{pyb::} every finite group $G$ contains an Abelian subgroup $H$ of size $\exp(\Omega(\sqrt{\log |G|}))$. (This is essentially best possible.)  Averaging, there is some $t \in G$ such that
\begin{equation*}
|tH \cap A | \geq \E_{t \in G}{|tH \cap A|} = \sum_{h \in H}{\E_{t \in G}{1_A(th)}} = \sum_{h \in H}{\frac{|A|}{|G|}} = \frac{|A||H|}{|G|}.
\end{equation*}
Since solutions to $z=yx^{-1}y$ are left invariant we conclude that $H \cap t^{-1}A$ contains no non-trivial solutions to $z=yx^{-1}y$.  However, this set is a subset of $H$ which is an Abelian group and so we can apply Bloom's result (Theorem \ref{thm.bloom}) to see that
\begin{equation*}
\frac{|A||H|}{|G|} \leq  |tH \cap A| = |H \cap t^{-1}A|= \frac{|H|}{\log^{1-o(1)}|H|} = \frac{|H|}{\log^{\frac{1}{2}-o(1)} |G|},
\end{equation*}
which can be rearranged to give the claimed bound.
\end{proof}

We shall prove Theorem \ref{thm.mainintro} by proving the following.
\begin{theorem}\label{thm.main}
Suppose that $G$ is a finite group, and $A \subset G$ has size $\alpha |G|$ and distinct squares \emph{i.e.} $a^2 \neq b^2$ if $a,b \in A$ are distinct.  Then the number of triples $(x,y,z) \in A^3$ such that $xz=y^2$ is $\exp(-\exp(\alpha^{-{O(1)}}))|G|^2$.
\end{theorem}
Theorem \ref{thm.mainintro} follows immediately from this.
\begin{proof}[Proof of Theorem \ref{thm.mainintro}]
If $a,b \in A$ have $a^2=b^2$ and $a \neq b$, then $(a,b,a)$ is a triple with $aa=b^2$ and $a \neq b$ and we are done.  It follows that we may assume $A$ has distinct squares and so we have a lower bound on the number of triples $(x,y,z) \in A^3$ such that $xz=y^2$.  By hypothesis we know that in this case $x=y$ whence $x=z$ and so the number of such triples is $|A|$ and hence
\begin{equation*}
\alpha |G|=|A|\geq \exp(-\exp(\alpha^{-{O(1)}}))|G|^2
\end{equation*}
which can be rearranged to give the result.
\end{proof}
The remainder of the paper is concerned with proving Theorem \ref{thm.main}.

\section{Notation}

Given a finite set $Z$ we write $M(Z)$ for the set of complex-valued measures on $Z$ and put
\begin{equation*}
\|\mu\|:=\int{d|\mu|} \text{ for all }\mu \in M(Z).
\end{equation*}
Suppose that $\mu$ is a non-negative measure supported on $Z$.  We write $L_p(\mu)$ for the space of functions $f:Z \rightarrow \C$ endowed with the (semi-)norm
\begin{equation*}
\|f\|_{L_p(\mu)}:=\left(\int{|f(z)|^pd\mu(z)}\right)^{1/p},
\end{equation*}
with the usual convention when $p=\infty$.  Of course $L_2(\mu)$ is a Hilbert space we we denote the inner product inducing the norm $\|\cdot\|_{L_2(\mu)}$ by
\begin{equation*}
\langle f,g\rangle_{L_2(\mu)}=\int{f(x)\overline{g(x)}d\mu(x)} \text{ for all }f,g \in L_2(\mu).
\end{equation*}
We shall often take $\mu$ to be the uniform probability measure supported on $Z$ which we denote $\mu_Z$.

Throughout the paper we work with sets in some finite group $G$.  The group structure will be encoded by the left and right regular representations defined on functions $f:G \rightarrow \C$ by
\begin{equation*}
\lambda_x(f)(y):=f(x^{-1}y) \text{ for all }y \in G;
\end{equation*}
and
\begin{equation*}
\rho_x(f)(y):=f(yx) \text{ for all }y \in G.
\end{equation*}
This extends to (complex-valued) measures, $\mu$ on $G$, where for each $x \in G$ we write $\rho_x(\mu)$ for the measure induced by
\begin{equation*}
\int{fd\rho_x(\mu)} = \int{\rho_{x^{-1}}(f)d\mu} \text{ for all }f:G \rightarrow \C,
\end{equation*}
and similarly for $\lambda_x(\mu)$.  Thus,
\begin{equation*}
\rho_x(\mu)(A)=\mu(Ax) \text{ and } \lambda_x(\mu)(A)=\mu(xA) \text{ for all }x\in G, A \subset G.
\end{equation*}

Given a (complex-valued) measure $\mu$ on $G$, and a function $f:G \rightarrow \C$ we define\begin{equation*}
\langle f, \mu\rangle := \int{fd\overline{\mu}} \text{ and } \langle \mu,f\rangle := \int{\overline{f}d\mu};
\end{equation*}
similarly define the \textbf{convolution} of $f$ and $\mu$ to be
\begin{equation*}
f\ast \mu(x):=\int{f(xy^{-1})d\mu(y)} \text{ and } \mu \ast f(x)=\int{f(y^{-1}x)d\mu(y)}\text{ for all }x \in G .
\end{equation*}
It is worth noting that if $A\subset G$ is non-empty then
\begin{equation*}
f \ast \mu_A(x) = \E_{a \in A}{f(xa^{-1})} = \E_{z \in xA^{-1}}{f(z)};
\end{equation*}
\emph{i.e.} $f \ast \mu_A(x)$ is the average value of $f$ on $xA^{-1}$.  Similarly, $\mu_A \ast f(x)$ is the average value of $f$ on $A^{-1}x$.

We shall also look to convolve two measures: suppose that $\mu$ and $\nu$ are such.  Then we define their convolution to be the measure induced by
\begin{equation*}
\int{f(z)d(\mu\ast \nu)(z)} = \int{f(xy)d\mu(x)d\mu(y)} \text{ for all }f :G \rightarrow \C.
\end{equation*}
Here it is worth noting that if $A,A' \subset G$ are non-empty then
\begin{equation*}
\supp \mu_A \ast \mu_{A'} = AA':=\{aa':a \in A, a' \in A'\}.
\end{equation*}

Finally, for $f:G \rightarrow \C$ define the
\begin{equation*}
\tilde{f}(x):=\overline{f(x^{-1})} \text{ for all }x \in G,
\end{equation*}
and similarly for measures.  Note that if $A \subset G$ is non-empty then $\widetilde{\mu_A} = \mu_{A^{-1}}$ where $A^{-1}:=\{a^{-1}: a \in A\}$.

We introduced the last piece of notation because it captures the adjoint operation.  In particular, the adjoint of $\mu \mapsto f \ast \mu$ is $\nu \mapsto \tilde{f}\ast\nu$ \emph{i.e.}
\begin{equation*}
\langle f \ast \mu,\nu\rangle = \langle \mu, \tilde{f}\ast \nu\rangle \text{ for all measures } \mu,\nu,
\end{equation*}
and, again, similarly for convolution with measures.  One can also use this notation to capture convolution:
\begin{equation*}
f \ast \mu(x) =\langle \rho_x(\mu), \tilde{f}\rangle \text{ and } \mu \ast f(x) = \langle \rho_x(f),\tilde{\mu}\rangle \text{ for all }x \in G,
\end{equation*}
and, similarly for $\lambda$.

\section{Multiplicative systems}\label{sec.ms}

Since the work of Roth \cite{rot::,rot::0} the standard approach to problems of the type considered in this paper has been inductive.  As often happens with inductive arguments they become possible when we enlarge the class we are working over.  In this case we shall prove our result not just for large subsets of groups, but for large subsets of certain group-like objects which we shall call multiplicative systems.

There is nothing particularly novel about the multiplicative systems presented in this paper and there are numerous essentially equivalent definitions extracting the key properties of a group which we require.  In the Abelian setting this has been explored extensively since the pioneering work of Bourgain \cite{bou::5}.  It may be worth noting that Gowers and Wolf use some nice notation in \cite{gowwol::} and Bloom \cite{blo::0} takes as basic a structure which is pretty close to ours.

Two of the axioms a group satisfies are particularly easy to guarantee for subsets of a group: we say that a set $A \subset G$ is a \textbf{symmetric neighbourhood of the identity} if it contains the identity and is closed under taking inverses \emph{i.e.} $1_G \in A$ and $x \in A$ implies $x^{-1} \in A$.  What is harder to capture is closure under multiplication and, indeed, we have to make do with a sort of `approximate' closure.  Given $r \in\N$ and $\epsilon\in [0,1]$ we say that
\begin{equation*}
\mathcal{B}=(B_{0+},B_0,B_{0-};B_{1+},B_1,B_{1-};\dots;B_{r+},B_{r},B_{r-};B_{r+1})
\end{equation*}
is an \textbf{$(r+1)$-step $\epsilon$-closed multiplicative system} if
\begin{enumerate}
\item \emph{(Symmetric neighbourhoods)} $B_{r+1}$, and $B_{i+}$, $B_i$, and $B_{i-}$, are symmetric neighbourhoods of the identity for all $0 \leq i \leq r$;
\item \emph{(Nesting)} we have
\begin{equation*}
B_{0+} \supset B_0 \supset B_{0-} \supset B_{1+} \supset B_ {1}\supset B_{1-}\supset \dots \supset B_{r+}\supset B_{r}\supset B_{r-}\supset B_{r+1};
\end{equation*}
\item \emph{(Closure)} we have
\begin{equation*}
B_{i-} \subset yB_ix \subset B_{i+} \text{ for all }x,y \in B_{i+1} \text{ for all } 0 \leq i \leq r,
\end{equation*}
and the estimates
\begin{equation*}
\frac{1}{1+\epsilon}|B_{i+}| \leq |B_i| \leq (1+\epsilon)|B_{i-}|\text{ for all } 0 \leq i \leq r.
\end{equation*}
\end{enumerate}
Note here that since the $B_i$s and $B_{i\pm}$s contain the identity, a number of the inclusions in nesting follow from closure.

The idea here is that an $r$-step system in enough to multiply group elements together about `$r$ times'.  The parameter $\epsilon$ captures the error each time we do this.  While we have set these systems up quite generally we shall only make use of systems with $r$ small, typically 1 or 2.

When we need more than one multiplicative system they will be denoted $\mathcal{B}'$, $\mathcal{B}''$ \emph{etc.} with the obvious convention that
\begin{equation*}
\mathcal{B}'=(B_{0+}',B_0',B_{0-}';B_{1+}',B_1',B_{1-}';\dots;B_{r+}',B_{r}',B_{r-}';B_{r+1}').
\end{equation*}
This is the reason that we have not chosen the easier-to-read notation $B_i^+$ and $B_i^-$ for $B_{i+}$ and $B_{i-}$.

The model we have in mind, and perhaps the simplest example, is given by groups.
\begin{example}[Groups]
Suppose that $H_{r+1} \leq H_r \leq \dots \leq H_1 \leq H_0 \leq G$.  Then
\begin{equation*}
(H_0,H_0,H_0;H_1,H_1,H_1;\dots;H_r,H_r,H_r;H_{r+1})
\end{equation*}
is an $(r+1)$-step $0$-closed multiplicative system.  
\end{example}
This example, and in fact the special case when $H_0=H_{r+1}$ is a very useful example to have in mind on a first reading of many of the results below.

In a $1$-step multiplicative system $\mathcal{B}$ we think of $B_1$ as `acting on' $B_0$, and there are many examples of multiplicative systems resulting from trivial action sets.
\begin{example}[Trivial action set]
For any symmetric neighbourhood of the identity $A$, we have that $(A,A,A;\{1_G\})$ is a $1$-step $0$-closed system.
\end{example}
The point of this example is just to emphasise that we shall be interested in the case when $B_{i+1}$ is not too much smaller than $B_i$.

To get a sense of how the various parameters behave it is useful to record some basic properties of multiplicative systems.
\begin{lemma}[Basic properties of multiplicative systems]
Suppose that $\mathcal{B}$ and $\mathcal{B}'$ are $(r+1)$-step (resp. $(r'+1)$-step) $\epsilon$-closed multiplicative systems.  Then
\begin{enumerate}
\item \emph{(Monotonicity)} $\mathcal{B}$ is an $(r+1)$-step $\epsilon''$-closed multiplicative system for all $\epsilon'' \geq \epsilon$;
\item \emph{(Truncation)} for $0 \leq l \leq m \leq r$ and any symmetric neighbourhood of the identity $B_* \subset B_{m+1}$,
\begin{equation*}
(B_{l+},B_l,B_{l-};\dots;B_{m+},B_{m},B_{m-};B_*)
\end{equation*}
is an $(m-l+1)$-step $\epsilon$-closed multiplicative system;
\item \emph{(Gluing)} if $B_{0+}' \subset B_{r+1}$ then
\begin{equation*}
(B_{0+},B_0,B_{0-};\dots;B_{r+},B_r,B_{r-};B_{0+}', B_0',B_{0-}'; \dots ; B_{r'+}',B_{r'}',B_{r'-}';B_{(r'+1)})
\end{equation*}
is an $(r+r'+1)$-step $\epsilon$-closed multiplicative system.
\end{enumerate}
\end{lemma}
At this stage we have not given any examples of multiplicative systems that are not either trivial or endowed with the far stronger structure of being a nested sequence of subgroups.  In Abelian groups we have a rich set of examples provided by Bohr sets.
\begin{example}[Bohr sets]\label{ex.bohr}
Throughout this example suppose that $G$ is an Abelian group and use additive rather than multiplicative notation for the group operation. Suppose that $\Gamma$ is a set of $d$ characters on $G$, and $\delta \in (0,2]$.  We define the \textbf{Bohr set with frequency set $\Gamma$ and width $\delta$} to be the set
\begin{equation*}
\Bohr(\Gamma,\delta):=\{x \in G: |\gamma(x)-1| \leq \delta \text{ for all }\gamma \in \Gamma\}.
\end{equation*}
Some fairly straight-forward arguments which can be found in \emph{e.g.} \cite[Lemma 4.19]{taovu::} show that for $l \in \N$ we have
\begin{equation*}
|\Bohr(\Gamma,\delta)| \leq l^{O(d)}|\Bohr(\Gamma,\delta/l)|
\end{equation*}
and an easy application of the triangle inequality shows us that
\begin{equation*}
l\Bohr(\Gamma,\delta/l):=\overbrace{\Bohr(\Gamma,\delta/l) +\dots + \Bohr(\Gamma,\delta/l)}^{l \text{ times.}} \subset \Bohr(\Gamma,\delta).
\end{equation*}
The ability to \emph{dilate} Bohr sets lets us use them to produce multiplicative systems.  In particular, we have
\begin{equation*}
l\Bohr(\Gamma,\delta/l)+\Bohr(\Gamma,\delta) \subset \Bohr(\Gamma,2\delta)
\end{equation*}
and so by the pigeonhole principle there is some $j < l$ such that
\begin{equation*}
|\Bohr(\Gamma,\delta/l)+(j\Bohr(\Gamma,\delta/l) +\Bohr(\Gamma,\delta))| \leq \exp(O(d/l))|j\Bohr(\Gamma,\delta/l) +\Bohr(\Gamma,\delta)|
\end{equation*}
For any $\epsilon \in (0,1]$ (the closure parameter of the system) we can pick $l=O(d\epsilon^{-1})$ such that $\exp(O(d/l)) \leq 1+\epsilon$, and so we have two sets 
\begin{equation*}
B_0:=j\Bohr(\Gamma,\delta/l) +\Bohr(\Gamma,\delta) \text{ and } B_1:=\Bohr(\Gamma,\delta/l)
\end{equation*}
such that $|B_1+B_0| \leq (1+\epsilon)|B_0|$.  It is a short step from here to defining a $2$-step $\epsilon$-closed multiplicative system.\footnote{We do not do it because, while $B_{0+}$ can just be defined to be $B_0+B_1$, there is a small technical obstacle to defining $B_{0-}$.  This is easy to resolve but detracts from the example.}  Crucially this multiplicative system satisfies
\begin{equation*}
|B_1| = \Omega(1/l)^{O(d)}|B_0|.
\end{equation*}
In fact in the case of Bohr sets a structure somewhat stronger than a multiplicative system can be constructed: we can arrange for a nested sequence of so-called regular Bohr sets.  This was originally done in \cite{bou::5}; an exposition may be found around \cite[Definition 4.23]{taovu::}.
\end{example}

There are natural analogues of Bohr sets in general finite groups, but they tend to describe only normal subsets of a group and that is not rich enough for our purposes.  We shall take a different approach to find a supply of multiplicative systems motivated by a result of Bogolio{\'u}boff \cite{bog::}.

Bogolio{\'u}boff showed that if $A$ is a symmetric subset of an Abelian group of density $\alpha$ then $4A:=A+A+A+A$ contains a large Bohr set.  Turning this around, in a general group we shall look for our multiplicative systems inside four-fold product set of large subsets of $G$.  Bogolio{\'u}boff's lemma was improved by Chang in \cite{cha::0}, and recently Croot and Sisask discovered a generalisation of Chang's argument to non-Abelian groups.

The following result is essentially \cite[Corollary 1.4]{crosis::} extended to functions.  We only need the version for sets here as it happens, but the proof for functions is no harder.  (We shall have to examine this later in Lemma \ref{lem.csrel} when we establish a version of the Croot-Sisask Lemma for multiplicative systems.)
\begin{lemma}[The Croot-Sisask Lemma]\label{lem.cs}  Suppose that $f \in L_p(\mu_G)$ for some $p \in [2,\infty)$, $X \subset G$ has density $\delta:=\mu_G(X)>0$, and $\eta \in (0,1]$ is a parameter.  Then the set of $x$ such that 
\begin{equation*}
\|\rho_{x^{-1}}(f \ast \mu_X) - f \ast \mu_X\|_{L_p(\mu_G)} \leq \eta \|f\|_{L_p(\mu_G)}
\end{equation*}
is a symmetric neighbourhood of the identity and has density $\exp(-O(\eta^{-2}p\log2\delta^{-1}))$.
\end{lemma}
We shall now use this to establish a Bogolio{\'u}boff-type lemma in general groups.  Before diving in, we should say that this result is just a variant of \cite[Theorem 1.6]{crosis::}.
\begin{lemma}\label{lem.nab}
Suppose that $X\subset G$ is a symmetric neighbourhood of the identity of density $\delta:=\mu_G(X)>0$, and $k \in \N$ is a parameter.  Then there is a symmetric neighbourhood of the identity $S$ such that
\begin{equation*}
S^k \subset X^4 \text{ and } \mu_G(S) \geq \exp(-O(k^2\log^22\delta^{-1})).
\end{equation*}
\end{lemma}
\begin{proof}
Let $p \geq 2$ and $\eta \in (0,1]$ be parameters to be chosen later and apply Lemma \ref{lem.cs} with $f=1_{X^2}$ to get a symmetric neighbourhood of the identity $S$ such that
\begin{equation*}
\|\rho_{x^{-1}}(1_{X^2} \ast \mu_{X}) - 1_{X^2} \ast \mu_{X}\|_{L_p(\mu_G)} \leq \eta \|1_{X^2}\|_{L_p(\mu_G)}.
\end{equation*}
By the triangle inequality we have that
\begin{equation*}
\|\rho_{x^{-1}}(1_{X^2} \ast \mu_{X}) - 1_{X^2} \ast \mu_{X}\|_{L_p(\mu_G)} \leq k\eta \|1_{X^2}\|_{L_p(\mu_G)},
\end{equation*}
for all $x \in S^k$, and so by H{\"o}lder's inequality we see that
\begin{equation*}
|\langle \rho_{x^{-1}}(1_{X^2} \ast \mu_X),1_X\rangle_{L_2(\mu_G)} - \langle 1_{X^2} \ast \mu_X,1_X\rangle_{L_2(\mu_G)} |  \leq k\eta |X|^{1-1/p}|X^2|^{1/p} \leq k\eta \delta^{-1/p}|X|
\end{equation*}
since $X^2 \subset G$.  On the other hand
\begin{equation*}
 \langle 1_{X^2} \ast \mu_X,1_X\rangle_{L_2(\mu_G)}=|X|,
\end{equation*}
so we can take $p=O(\log 2\delta^{-1})$ and $\eta = \Omega(1/k)$ such that
\begin{equation*}
\langle \rho_{x^{-1}}(1_{X^2} \ast \mu_X),1_X\rangle_{L_2(\mu_G)}  >|X|/2 \text{ for all } x\in S^k.
\end{equation*}
But then the result follows since
\begin{equation*}
\langle \rho_{x^{-1}}(1_{X^2} \ast \mu_X),1_X\rangle_{L_2(\mu_G)}= 1_X \ast 1_{X^2} \ast \mu_X(x^{-1}),
\end{equation*}
and $\supp 1_X \ast 1_{X^2} \ast \mu_X \subset X^4$.
\end{proof}
If $G$ were Abelian, and $X$ a Bohr set then it would be possible to take $S$ with
\begin{equation*}
|S| \geq \exp(-O((\log 2k)( \log 2\delta^{-1}))),
\end{equation*}
and so the $\delta$-dependence in the above is not too bad even though the $k$ dependence is rather poor.  Fortunately in our applications $k$ will tend to be fixed, while $\delta$ will decrease.

Finally we can use this lemma to produce some multiplicative systems.  The argument is essentially the same pigeon-hole as we did with Bohr sets in Example \ref{ex.bohr}, replacing the nice properties of dilates of Bohr sets by applications of Lemma \ref{lem.nab}.
\begin{corollary}\label{cor.cont}
Suppose that $X$ is a symmetric neighbourhood of the identity in $G$ with $\delta:=\mu_G(X) >0$, and $r \in \N_0$ and $\epsilon \in (0,1]$ are parameters.  Then there is an $(r+1)$-step $\epsilon$-closed multiplicative system $\mathcal{B}$ and some symmetric neighbourhood of the identity $S$ such that
\begin{equation*}
B_{0+} \subset X^4, S^4\subset B_{r+1} \text{ and } \mu_G(S) \geq \exp(-O((\epsilon^{-2}\log 2\delta^{-1})^{4^{r+1}})).
\end{equation*}
\end{corollary}
\begin{proof}
First apply Lemma \ref{lem.nab} to get a symmetric neighbourhood of the identity $S_0$ such that
\begin{equation*}
S_0^9 \subset X^4 \text{ and } \mu_G(S_0) \geq \exp(-O(\log^{2}2\delta^{-1})).
\end{equation*}
We shall now proceed inductively to construct sequences $((B_{l+},B_l,B_{l-}))_{l=0}^{r}$ and $(S_l)_{l=0}^{r+1}$ with
\begin{equation*}
S_{i+1}^9 \subset B_{i-}, \subset B_{i+} \subset S_i^9 \text{ and } B_{i-} \subset xB_iy\subset B_{i+} \text{ for all }x,y \in S_{i+1}^9,
\end{equation*}
and
\begin{equation*}
\frac{1}{1+\epsilon}\mu_G(B_{i+})\leq \mu_G(B_i)\leq (1+\epsilon)\mu_G(B_{i-}),
\end{equation*}
and, writing $\delta_i:=\mu_G(S_i)$, such that
\begin{equation*}
\delta_{i+1}\geq \exp(-O(\epsilon^{-2}\log^42\delta_i^{-1})).
\end{equation*}
Suppose we have constructed $S_i$, but not $B_{i+},B_i,B_{i-}$ or $S_{i+1}$.  Apply Lemma \ref{lem.nab} to the set $S_i$ with a natural $l_i$ to be chosen later to get a symmetric neighbourhood of the identity $S_{i+1}$ such that
\begin{equation*}
S_{i+1}^{l_i} \subset S_i^4 \text{ and } \delta_{i+1} \geq \exp(-O(l_i^2\log^22\delta_i^{-1})).
\end{equation*}
Now
\begin{equation*}
\mu_G(S_{i+1}^{l_i}S_i S_{i+1}^{l_i}) \leq 1= \delta_i^{-1}\mu_G(S_i),
\end{equation*}
and so, by telescoping products,
\begin{equation*}
\prod_{j=1}^{\lfloor l_i/18\rfloor-1 }{\frac{\mu_G(S_{i+1}^{18(j+1)}S_i S_{i+1}^{18(j+1)})}{\mu_G(S_{i+1}^{18j}S_i S_{i+1}^{18j})}} \leq \delta_i^{-1}.
\end{equation*}
Thus we can take $l_i = O(\epsilon^{-1}\log 2\delta_i^{-1})$ so that there is some $0<j <\lfloor l_i/18\rfloor$ with
\begin{equation*}
\mu_G(S_{i+1}^{18} (S_{i+1}^{18j}S_i S_{i+1}^{18j})S_{i+1}^{18}) \leq (1+\epsilon)\mu_G(S_{i+1}^{18j}S_i S_{i+1}^{18j}).
\end{equation*}
We put
\begin{equation*}
B_{i-}:=S_{i+1}^{18j}S_i S_{i+1}^{18j}, B_i:=S_{i+1}^9B_{i-}S_{i+1}^9, \text{ and }B_{i+}:=S_{i+1}^9B_{i}S_{i+1}^9
\end{equation*}
and note that $B_{i+}$, $B_i$ and $B_{i-}$ are all symmetric neighbourhoods of the identity and
\begin{equation*}
B_{i-} \subset xB_iy\subset B_{i+} \text{ for all }x,y \in S_{i+1}^9.
\end{equation*}
Furthermore we have $\mu_G(B_{i+})\leq (1+\epsilon)\mu_G(B_{i-})$ and so
\begin{equation*}
\frac{1}{1+\epsilon}\mu_G(B_{i+})\leq \mu_G(B_i)\leq (1+\epsilon)\mu_G(B_{i-}).
\end{equation*}
Finally
\begin{equation*}
S_{i+1}^9 \subset B_{i-} \text{ and }B_{i+} \subset S_{i+1}^{18j+18}S_iS_{i+1}^{18j+18} \subset S_{i+1}^{l_i}S_iS_{i+1}^{l_i} \subset S_i^9,
\end{equation*}
and
\begin{equation*}
\delta_{i+1} \geq  \exp(-O(\epsilon^{-2}\log^42\delta_i^{-1})).
\end{equation*}
We complete the construction by putting $S:=S_{r+1}$ and $B_{r+1}:=S^4$, and initialise the construction with $S_0$ as described above.  The result follows on working out the bounds for $\delta_{r+1}$.
\end{proof}
It is possible to relate $X$ to a much more rigid structure called a coset nilprogression.  These are somewhat complicated to define and the interested reader is directed to the paper \cite{bregretao::} Breuillard, Green and Tao, where this and related matters are addressed although at the expense of bounds.

At this stage we have established the results we shall need for the generation of suitable multiplicative systems and we can turn to tools for using them.  First note the simple observation that conjugation preserves multiplicative systems.
\begin{lemma}[Conjugation of multiplicative systems]
Suppose that $\mathcal{B}$ is an $(r+1)$-step $\epsilon$-closed multiplicative system and $g \in G$.  Then
\begin{equation*}
(gB_{0+}g^{-1},gB_0g^{-1},gB_{0-}g^{-1};\dots;gB_{r+}g^{-1},gB_rg^{-1},gB_{r-}g^{-1};gB_{r+1}g^{-1})
\end{equation*}
is also an $(r+1)$-step $\epsilon$-closed multiplicative system.
\end{lemma}
As will be clear from what we have already written, we shall find our multiplicative systems inside four-fold product sets.  The following lemma will help us combine this with conjugation.
\begin{lemma}\label{lem.int}
Suppose that $S\subset G$ is a symmetric set with density $\sigma:=\mu_G(S)>0$ and $g,h \in G$.  Then there is a symmetric neighbourhood of the identity, $X$, such that
\begin{equation*}
X^4 \subset (gS^4g^{-1})\cap (hS^4h^{-1}) \text{ and }\mu_G(X) \geq \exp(-O(\log^22\sigma^{-1})).
\end{equation*}
\end{lemma}
\begin{proof}
We apply Lemma \ref{lem.nab} to get a symmetric neighbourhood of the identity $R$, such that $R^8 \subset S^4$, and $\mu_G(R) \geq \exp(-O(\log^2 2\sigma^{-1}))$.  Now note that
\begin{eqnarray*}
\mu_G((gR^2g^{-1})\cap (hR^2h^{-1})) \mu_G(R)^2 & \geq & \langle 1_{gR} \ast \widetilde{1_{gR}},1_{hR} \ast \widetilde{1_{hR}}\rangle_{L_2(\mu_G)}\\ & = & \|\widetilde{1_{hR}}\ast 1_{gR}\|_{L_2(\mu_G)}^2\\ & \geq & \|\widetilde{1_{hR}}\ast 1_{gR}\|_{L_1(\mu_G)}^2\\ & =& (\mu_G(hR)\mu_G(gR))^2 = \mu_G(R)^4,
\end{eqnarray*}
and so
\begin{equation*}
\mu_G((gR^2g^{-1})\cap (hR^2h^{-1})) \geq \mu_G(R)^2 = \exp(-O(\log^2 2\sigma^{-1})).
\end{equation*}
On the other hand
\begin{equation*}
(gS^4g^{-1})\cap (hS^4h^{-1}) \supset (gR^8g^{-1})\cap (hR^8h^{-1}) \supset (gR^2g^{-1} \cap hR^2h^{-1})^4,
\end{equation*}
and $(gR^2g^{-1})\cap (hR^2h^{-1})$ is also a symmetric neighbourhood of the identity.  The result follows on letting $X$ be this set.
\end{proof}

At this point we turn to analysis on multiplicative systems.  Just as analysis on finite groups begins with Haar measure -- the unique translation invariant probability measure on $G$ -- analysis on multiplicative systems begins with an approximate version of this.
\begin{lemma}[Approximate right invariant Haar measure]\label{lem.rhaar}
Suppose that $\mathcal{B}$ is an $(r+1)$-step $\epsilon$-closed multiplicative system.  Then
\begin{equation*}
\|\rho_{x^{-1}}(\mu_{B_i}) - \mu_{B_i}\| = O(\epsilon) \text{ for all } x \in B_{i+1},
\end{equation*}
where $0 \leq i \leq r$.
\end{lemma}
\begin{proof}
Just note that
\begin{eqnarray*}
\|\rho_{x^{-1}}(\mu_{B_i}) - \mu_{B_i}\|  &= & \mu_{B_i}(B_i \setminus B_ix) + \mu_{B_ix}(B_ix \setminus B_i)\\ & \leq & \mu_{B_i}(B_i \setminus xB_{i-}) + \mu_{B_i}(B_i \setminus B_{i-}x^{-1})\\ & \leq & 2-2\mu_{B_i}(B_{i-}) = O(\epsilon).
\end{eqnarray*}
The result is proved.
\end{proof}
We shall need a number of results in both a left and right hand form.  Typically we shall prove the right hand version and simply state the left hand version after it, the proof being essentially the same.
\begin{lemma}[Approximate left invariant Haar measure]
Suppose that $\mathcal{B}$ is an $(r+1)$-step $\epsilon$-closed multiplicative system.  Then
\begin{equation*}
\|\lambda_x(\mu_{B_i}) - \mu_{B_i}\| = O(\epsilon) \text{ for all } x \in B_{i+1}
\end{equation*}
where $0 \leq i \leq r$.
\end{lemma}

Our arguments will involve finding (smaller and smaller) multiplicative systems on which our original set has larger and larger density.  There are various different ways of doing this in Abelian groups and many of them have analogues in our setting.  We shall use the energy increment technique developed by Heath-Brown and Szemer{\'e}di in \cite{hea::} and \cite{sze::2} and record an appropriate version now.
\begin{lemma}\label{lem.l2right}
Suppose that $\mathcal{B}$ is a $1$-step $\epsilon$-closed multiplicative system; $A \subset Z:=gB_0$ has density $\alpha:=\mu_Z(A)>0$; and
\begin{equation*}
\|1_A \ast \mu_{B_1}-\alpha 1_Z\|_{L_2(\mu_Z)}^2 \geq \eta \alpha^2.
\end{equation*}
Then there is some $z \in Z$ such that
\begin{equation*}
\mu_{zB_1}(A)=1_A \ast \mu_{B_1}(z)  \geq \alpha(1+\eta) - O(\epsilon).
\end{equation*}
\end{lemma}
\begin{proof}
This is just a calculation.  First note that
\begin{eqnarray}
\nonumber \|1_A \ast \mu_{B_1}\|_{L_2(\mu_Z)}^2&=&\|1_A \ast \mu_{B_1}-\alpha 1_Z\|_{L_2(\mu_Z)}^2\\ \label{eqn.l2} & &  + \alpha \langle 1_A \ast \mu_{B_1},1_Z\rangle_{L_2(\mu_Z)} + \alpha \langle 1_Z,1_A \ast \mu_{B_1}\rangle_{L_2(\mu_Z)} - \alpha^2.
\end{eqnarray}
Now Lemma \ref{lem.rhaar} and the integral triangle inequality tell us that
\begin{eqnarray*}
\|\mu_Z \ast \mu_{B_1} - \mu_Z\| & = & \|\mu_{B_0} \ast \mu_{B_1} - \mu_{B_0}\| \\ & \leq & \int{\|\rho_{x^{-1}}(\mu_{B_0}) - \mu_{B_0}\|d\mu_{B_1}(x)}=O(\epsilon).
\end{eqnarray*}
Hence
\begin{equation*}
\langle 1_A \ast \mu_{B_1},1_Z\rangle_{L_2(\mu_Z)} = \langle 1_A \ast \mu_{B_1},\mu_Z\rangle = \langle 1_A, \mu_Z\ast \mu_{B_1}\rangle = \alpha+O(\epsilon),
\end{equation*}
and similarly
\begin{equation*}
\langle 1_Z,1_A \ast \mu_{B_1}\rangle_{L_2(\mu_Z)} = \alpha + O(\epsilon).
\end{equation*}
Inserting these into (\ref{eqn.l2}) we get
\begin{equation*}
\|1_A \ast \mu_{B_1}\|_{L_2(\mu_Z)}^2=\|1_A \ast \mu_{B_1}-\alpha 1_Z\|_{L_2(\mu_Z)}^2+\alpha^2+O(\epsilon \alpha) \geq (1+\eta)\alpha^2 - O(\epsilon \alpha)
\end{equation*}
On the other hand we have
\begin{equation*}
\|1_A \ast \mu_{B_1}\|_{L_2(\mu_Z)}^2\leq \|1_A \ast \mu_{B_1}\|_{L_\infty(\mu_Z)}\alpha,
\end{equation*}
by the triangle inequality and the result follows.
\end{proof}
In exactly the same way we have a left hand version of the above.
\begin{lemma}\label{lem.l2left}
Suppose that $\mathcal{B}$ is a $1$-step $\epsilon$-closed multiplicative system; $A \subset Z:=B_0h^{-1}$ has density $\alpha:=\mu_Z(A)>0$; and
\begin{equation*}
\|\mu_{B_1} \ast 1_A-\alpha 1_Z\|_{L_2(\mu_Z)}^2 \geq \eta \alpha^2.
\end{equation*}
Then there is some $z \in Z$ such that
\begin{equation*}
\mu_{B_1z}(A)= \mu_{B_1}\ast 1_A(z)  \geq \alpha(1+\eta) - O(\epsilon).
\end{equation*}
\end{lemma}

We shall ultimately be interested in examining some fairly thin subsets of multiplicative systems.  In Abelian groups the tool for doing this is Chang's lemma \cite{cha::0}; as we mentioned before, in non-Abelian groups the tool is the Croot-Sisask Lemma.  We shall actually need a version for functions on multiplicative systems, where the proof is a minor variant on that of Lemma \ref{lem.cs} with a few technicalities resulting from the fact that $\rho$ is not an isometry when restricted to multiplicative systems.
\begin{lemma}[The (right hand) Croot-Sisask Lemma for multiplicative systems]\label{lem.csrel}  Suppose that $\mathcal{B}$ is a $2$-step $\epsilon$-closed multiplicative system, $X$ is a symmetric neighbourhood of the identity and $X^8 \subset B_2$, $f \in L_\infty(\mu_{B_{0+}})$, $A \subset B_{0-}$ has density $\alpha:=\mu_{B_0}(A)>0$, and $\eta \in (0,1]$ and $p \in [2,\infty)$ are parameters.  Then there is a symmetric neighbourhood of the identity $T \subset X^2$ with
\begin{equation*}
\mu_G(T) \geq \exp(-O(\eta^{-2}p\log 2\alpha^{-1}))\mu_G(X)
\end{equation*}
such that
\begin{equation*}
\|\rho_{t^{-1}}(f \ast \mu_A) - f \ast \mu_A\|_{L_p(\mu_{B_1})} \leq \eta(1+O(\epsilon/p)) \|f\|_{L_\infty(\mu_{B_{0+}})}
\end{equation*}
for all $t \in T^4$.
\end{lemma}
\begin{proof}
Let $z_1,\dots,z_k$ be independent uniformly distributed $A$-valued random variables, and for each $y \in B_{1+}$ define $Z_i(y):=\rho_{z_i^{-1}}(f)(y) - f \ast \mu_A(y)$.  For fixed $y$, the variables $Z_i(y)$ are independent and have mean $0$, so it follows by the Marcinkiewicz-Zygmund inequality and then H{\"o}lder's inequality that
\begin{eqnarray*}
\| \sum_{i=1}^k{Z_i(y)}\|_{L^p(\mu_A^k)}^p &= &O(p)^{p/2}\int{\left(\sum_{i=1}^k{|Z_i(y)|^2}\right)^{p/2}d\mu_A^k}\\ & = & O(p)^{p/2}k^{p/2-1}\sum_{i=1}^k{\int{|Z_i(y)|^p}d\mu_A^k}.
\end{eqnarray*}
Integrating over $y\in B_{1+}$ and interchanging the order of summation we get
\begin{equation}\label{eqn.khin}
\int{\| \sum_{i=1}^k{Z_i(y)}\|_{L^p(\mu_A^k)}^pd\mu_{B_{1+}}(y)} = O(p)^{p/2}k^{p/2-1}\int{\sum_{i=1}^k{\int{|Z_i(y)|^pd\mu_{B_{1+}}(y)}}d\mu_{A}^k}.
\end{equation}
On the other hand,
\begin{eqnarray*}
\left(\int{|Z_i(y)|^pd\mu_{B_{1+}}(y)}\right)^{1/p} & = & \|Z_i\|_{L_p(\mu_{B_{1+}})}\\ & \leq & \|\rho_{z_i^{-1}}(f)\|_{L_p(\mu_{B_{1+}})} + \|f \ast \mu_A\|_{L_p(\mu_{B_{1+}})} \leq 2\|f\|_{L_\infty(\mu_{B_{0+}})}
\end{eqnarray*}
by the triangle inequality and the fact that $A\subset B_{0-}$.  Dividing (\ref{eqn.khin}) by $k^p$ and inserting the above and the expression for the $Z_i$s we get that
\begin{equation*}
\int{\int{\left|\frac{1}{k}\sum_{i=1}^k{\rho_{z_i^{-1}}(f)(y)} - f \ast \mu_A(y)\right|^pd\mu_{B_{1+}}(y)}d\mu_A^k(z)}=O(pk^{-1}\|f\|_{L_\infty(\mu_{B_{0+}})}^2)^{p/2}.
\end{equation*}
Pick $k=O(\eta^{-2}p)$ such that the right hand side is at most $(\eta \|f\|_{L_\infty(\mu_{B_{0+}})}/16)^p$ and write $\mathcal{L}$ for the set of $x \in A\times \dots \times A$ (where the product is $k$-fold) for which the integrand above is at most $(\eta \|f\|_{L_\infty(\mu_{B_{0+}})}/8)^p$.  By averaging $\mu_A^k(\mathcal{L}^c) \leq 2^{-p}$ and so $\mu_A^k(\mathcal{L}) \geq 1-2^{-p} \geq 3/4$.

Now, put $\Delta:=\{(x,\dots,x): x \in X\}$ and note (since $A\subset B_{0-}$ implies $\mathcal{L} \subset B_{0-}\times \dots \times B_{0-}$ and $\mathcal{L}+\Delta \subset B_0\times \dots \times B_0$ ) that
\begin{eqnarray*}
\langle \widetilde{1_{\mathcal{L}}}\ast \mu_{\mathcal{L}}, \mu_\Delta\ast \widetilde{\mu_\Delta}\rangle & = & \langle \mu_{\mathcal{L}}\ast \mu_\Delta, 1_{\mathcal{L}} \ast \mu_\Delta\rangle\\ & = & \|1_\mathcal{L} \ast \mu_{\Delta}\|_{L_2(\mu_{B_0}^k)}\mu_{B_0}^k(\mathcal{L})^{-1}\\ & \geq & \|1_{\mathcal{L}}\ast \mu_\Delta\|_{L_1(\mu_{B_0}^k)}^2\mu_{B_0}^k(\mathcal{L})^{-1} =\mu_{B_0}^k(\mathcal{L}) \geq \frac{3}{4}\alpha^{k}.
\end{eqnarray*}
We let $T \subset X^2$ be the set of $t$ such that $\widetilde{1_{\mathcal{L}}}\ast \mu_{\mathcal{L}}(t,\dots,t)>0$.  Then
\begin{equation*}
\mu_\Delta\ast \widetilde{\mu_\Delta}(\{(t,\dots,t):t \in T\}) \geq \frac{3}{4}\mu_{B_0}(A)^{k},
\end{equation*}
and so $\mu_{G}(T) \geq \frac{3}{4}\alpha^{k}\mu_G(X)$.

Now suppose that $t_1,t_2,t_3,t_4 \in T$.  Then for each $1 \leq j \leq 4$ there are elements $z(t_j),y(t_j) \in \mathcal{L}$ such that $y(t_j)_i=z(t_j)_it_j$ for all $1 \leq i \leq k$.  By the triangle inequality
\begin{eqnarray}
\label{eqn.f} \|\rho_{(t_1t_2t_3t_4)^{-1}}(f \ast \mu_A) - f \ast \mu_A\|_{L_p(\mu_{B_1})} & \leq &  \|\rho_{(t_2t_3t_4)^{-1}}(\rho_{t_1^{-1}}(f \ast \mu_A) - f \ast \mu_A)\|_{L_p(\mu_{B_1})}\\ \nonumber & + &  \|\rho_{(t_3t_4)^{-1}}(\rho_{t_2^{-1}}(f \ast \mu_A) - f \ast \mu_A)\|_{L_p(\mu_{B_1})}\\ \nonumber & + &  \|\rho_{t_4^{-1}}(\rho_{t_3^{-1}}(f \ast \mu_A) - f \ast \mu_A)\|_{L_p(\mu_{B_1})}\\ \nonumber & + &  \|\rho_{t_4^{-1}}(f \ast \mu_A) - f \ast \mu_A\|_{L_p(\mu_{B_1})}\\ \nonumber& \leq & \sup_{x \in T^3, t \in T}{\|\rho_{t^{-1}}(f \ast \mu_A) - f \ast \mu_A\|_{L_p(\rho_x(\mu_{B_1}))}}
\end{eqnarray}
Now, suppose that $x \in T^3$ and $t \in T$.  Then
\begin{eqnarray*}
\|\rho_{t^{-1}}(f \ast \mu_A) - f \ast \mu_A\|_{L_p(\rho_x(\mu_{B_1}))}& \leq &\|\rho_{t^{-1}}\left(\frac{1}{k}\sum_{i=1}^k{\rho_{z(t)_i^{-1}}(f)}\right) -  f \ast \mu_A\|_{L_p(\rho_x(\mu_{B_1}))}\\ &&+\|\rho_{t^{-1}}\left(\frac{1}{k}\sum_{i=1}^k{\rho_{z(t)_i^{-1}}(f)} - f \ast \mu_A \right)\|_{L_p(\rho_x(\mu_{B_1}))}\\ & = &\|\frac{1}{k}\sum_{i=1}^k{\rho_{y(t)_i^{-1}}(f)} -  f \ast \mu_A\|_{L_p(\rho_x(\mu_{B_1}))}\\ &&+\|\frac{1}{k}\sum_{i=1}^k{\rho_{z(t)_i^{-1}}(f)} - f \ast \mu_A \|_{L_p(\rho_{tx}(\mu_{B_1}))}.
\end{eqnarray*}
However, since $t \in T$ and $x\in T^3$ we have $x, tx \in T^4 \subset X^8 \subset B_2$, hence $\rho_x(\mu_{B_1}) \leq (1+O(\epsilon))\mu_{B_{1+}}$ and so
\begin{eqnarray*}
\|\rho_{t^{-1}}(f \ast \mu_A) - f \ast \mu_A\|_{L_p(\rho_x(\mu_{B_1}))}& \leq &(1+O(\epsilon/p))\left(\|\frac{1}{k}\sum_{i=1}^k{\rho_{y(t)_i^{-1}}(f)} -  f \ast \mu_A\|_{L_p(\mu_{B_{1+}})}\right.\\ &&\left.+\|\frac{1}{k}\sum_{i=1}^k{\rho_{z(t)_i^{-1}}(f)} - f \ast \mu_A \|_{L_p(\mu_{B_{1+}})}\right)\\ & \leq & (1+O(\epsilon/p))\eta\|f\|_{L_\infty(\mu_{B_{0+}})}/4.
\end{eqnarray*}
The last inequality is from the definition of $\mathcal{L}$.  Inserting this bound into (\ref{eqn.f}) gives us the required result.
\end{proof}
\begin{lemma}[The (left hand) Croot-Sisask Lemma for multiplicative systems]\label{lem.csrelleft}  Suppose that $\mathcal{B}$ is a $2$-step $\epsilon$-closed multiplicative system, $X$ is a symmetric neighbourhood of the identity and $X^8 \subset B_2$, $f \in L_\infty(\mu_{B_{0+}})$, $A \subset B_{0-}$ has density $\alpha:=\mu_{B_0}(A)>0$, and $\eta \in (0,1]$ and $p \in [2,\infty)$ are parameters.  Then there is a symmetric neighbourhood of the identity $T \subset X^2$ with
\begin{equation*}
\mu_G(T) \geq \exp(-O(\eta^{-2}p\log 2\alpha^{-1}))\mu_G(X)
\end{equation*}
such that
\begin{equation*}
\|\lambda_{t}(\mu_A\ast f) - \mu_A \ast f\|_{L_p(\mu_{B_1})} \leq \eta(1+O(\epsilon/p)) \|f\|_{L_\infty(\mu_{B_{0+}})}
\end{equation*}
for all $t \in T^4$.
\end{lemma}

\section{The iteration lemmas}\label{sec.u1}

Having set up the basic machinery in the previous section, this section is devoted to establishing the key iteration lemmas that are specific to the problem we are considering here.  All of the results will have the form of a dichotomy: either we shall find some sort of density increment on a new multiplicative system; or we shall be able to ensure some good behaviour. 

There are two key results, Proposition \ref{prop.u1} and Proposition \ref{prop.u2}.  These results correspond roughly to making some relative versions of the $U^1$-norm and $U^2$-norm small respectively.  The first of these is rather easy and in the Abelian case is essentially \cite[\S5]{bou::5}.  This (in the application of Lemma \ref{lem.squares} inside the proof of Proposition \ref{prop.u1}) is where the requirement that the elements of $A$ have distinct squares comes from.

To understand the argument it may be useful to consider a model example.  Meshulam's proof for Roth's Theorem in $\F_3^n$ \cite{mes::} can be viewed (somewhat anachronistically)  as a model version of Bourgain's proof of Roth's Theorem in $\Z$ \cite{bou::5} in which all the Bohr sets are assumed to be subgroups.  The same modelling assumption is useful here.

Suppose that $A \subset gHk^{-1}$ for some $H \leq G$ and elements $g,k \in G$ has size $\alpha |H|$.  Write $K:=gHg^{-1}\cap kHk^{-1}$.  It is possible to show by averaging that (either we have a density increment or) the set
\begin{equation*}
S:=\{a \in A: \mu_K\ast 1_{Agk^{-1}} (akg^{-1}) \approx \alpha \text{ and } 1_{gk^{-1}A}\ast \mu_K (gk^{-1}a) \approx \alpha\},
\end{equation*}
has $|S| \approx \alpha |H|$.  This is the application of Lemmas \ref{lem.equaliseright} and \ref{lem.equaliseleft} below.

By a slightly more complicated averaging argument (Lemma \ref{lem.squares}) there is some $a \in S$ such that
\begin{equation*}
sK=aK \text{ and }Ks=Ka \text{ for at least } \frac{|K|^2}{|KS||SK|}|S| \text{ elements } s \in S.
\end{equation*}
Since $|KS| \leq |H|$ and $|SK| \leq |H|$ it follows that
\begin{equation*}
|\{s^2:s \in S\} \cap aKa\}| \gtrsim \alpha \left(\frac{|K|}{|H|}\right) |K|,
\end{equation*}
and then we see that
\begin{equation*}
A_1:=Aa^{-1} \cap K, A_2 :=a^{-1}A \cap K \text{ and }T:=a^{-1}\{s^2: s \in A\}a^{-1}\cap K
\end{equation*}
have
\begin{equation*}
|A_1| \approx \alpha |K|, |A_2| \approx \alpha |K| \text{ and }|T| \gtrsim \alpha \left(\frac{|K|}{|H|}\right),
\end{equation*}
and we have an injection from triples $(a_1,a_2,t) \in A_1\times A_2\times T$ with $a_2a_1=t$ to triples $(a_1a,aa_2,ata) \in A \times A \times \{s^2:s \in A\}$ with $(aa_2)(a_1a)=ata$.

Counting triples $(a,b,c) \in A \times B\times C$ such that $ab=c$ is actually rather well understood in simple groups as was shown by Gowers in \cite{gow::2}, but in Abelian groups it leads to a dichotomy: either the count of triples is about right or else there is a structure (for us a multiplicative system) on which at least one of the sets has increased density.  This dichotomy also holds for general groups and is out Proposition \ref{prop.u2}.

The tool for proving the analogue of Proposition \ref{prop.u2} in the Abelian setting is the Fourier transform (and Chang's theorem) and for us here we require the Croot-Sisask Lemma (as mentioned in \S\ref{sec.ms}).
\begin{proposition}\label{prop.u1}
Suppose that $\mathcal{B}$ and $\mathcal{B}'$ are $1$-step $\epsilon$-closed multiplicative systems with $B'_0\subset gB_1g^{-1} \cap h B_1 h^{-1}$; $A \subset Z:=gB_0h^{-1}$ has $\alpha:=\mu_{Z}(A) >0$ (and distinct squares); $X$ is a symmetric neighbourhood of the identity with $\delta:=\mu_G(X)>0$ such that $X^4 \subset B_1'$; and $\eta \in (0,1]$ is a parameter.  Then
\begin{enumerate}
\item either there is some $a \in gB_0h^{-1}$ such that
\begin{equation*}
\mu_{aB_0'}(A), \mu_{B_0'a}(A) \geq \alpha(1-\eta)
\end{equation*}
and
\begin{equation*}
\mu_{aB_1'a}(\{s^2:s \in A\}) =\Omega(\alpha \delta^2);
\end{equation*}
\item or there is some $z \in gB_0h^{-1}$ such that
\begin{equation*}
\mu_{B_1'}\ast 1_A(z) \geq \alpha(1+\Omega(\eta^2)) - O(\epsilon\alpha^{-1});
\end{equation*}
\item or  there is some $z \in gB_0h^{-1}$ such that
\begin{equation*}
1_A\ast \mu_{B_1'}(z) \geq \alpha(1+\Omega(\eta^2)) - O(\epsilon\alpha^{-1}).
\end{equation*}
\end{enumerate}
\end{proposition}
First we have two lemmas (which are left and right versions of each other) which will be used in producing the second and third outcomes above.
\begin{lemma}\label{lem.equaliseright}
Suppose $\mathcal{B}$ and $\mathcal{B}'$ are $1$-step $\epsilon$-closed multiplicative systems with $B_0' \subset hB_1h^{-1}$; and $A \subset Z:=gB_0h^{-1}$ has density $\alpha:=\mu_{Z}(A)>0$; and $\eta \in (0,1]$ is a parameter.  Then
\begin{enumerate}
\item either we have
\begin{equation*}
\|1_A \ast \mu_{B_0'}-\alpha1_{Z}\|_{L_1(\mu_A)} \leq \eta\alpha;
\end{equation*}
\item or there is some $z \in Z$ such that
\begin{equation*}
1_A \ast \mu_{B_1'}(z) \geq \alpha(1+\Omega(\eta^2))-O(\epsilon\alpha^{-1}).
\end{equation*}
\end{enumerate}
\end{lemma}
\begin{proof}
Write
\begin{equation*}
f(z):=\begin{cases} |1_A \ast \mu_{B_0'}(z) - \alpha 1_Z(z)| & \text{ for all }z \in Z=gB_0h^{-1}\\
0 & \text{ otherwise.}\end{cases}
\end{equation*}
Since
\begin{equation*}
\|\rho_{x^{-1}}(1_A \ast \mu_{B_0'}) - 1_A \ast \mu_{B_0'}\|_{L_\infty(\mu_{G})} =O(\epsilon) \text{ for all }x \in B_1',
\end{equation*}
by the closure properties of $\mathcal{B}'$, and $\rho_{x^{-1}}(1_{Z})(z)=1_Z(z)$ for all $x \in B_1'$ and $z \in gB_{0-}h^{-1}$ (since $B_1' \subset B_0' \subset hB_1h^{-1}$), we have
\begin{equation*}
|\rho_{x^{-1}}(f)(z) - f(z)| = O(\epsilon) \text{ for all } z \in gB_{0-}h^{-1} \text{ and } x \in B_1'.
\end{equation*}
It follows that
\begin{equation*}
|\langle f,1_A\rangle_{L_2(\mu_Z)} -\langle \rho_{x^{-1}}(f),1_A\rangle_{L_2(\mu_Z)}| = O(\epsilon\alpha) + O(\mu_Z(Z \setminus gB_{0-}h^{-1}))= O(\epsilon).
\end{equation*}
We conclude that if we are not in the first case of the lemma then we have
\begin{equation*}
\langle f,1_A\ast \mu_{B_1'} \rangle_{L_2(\mu_Z)}=\langle f\ast \mu_{B_1'},1_A \rangle_{L_2(\mu_Z)} > \eta \alpha^2 - O(\epsilon),
\end{equation*}
and hence that
\begin{equation*}
\langle f,1_A\ast \mu_{B_1'} -\alpha 1_{Z}\rangle_{L_2(\mu_Z)} + \alpha \|f\|_{L_1(\mu_Z)} > \eta \alpha^2 - O(\epsilon).
\end{equation*}
Applying the Cauchy-Schwarz inequality to the first inner product, and nesting of norms to $\|f\|_{L_1(\mu_Z)}$ this then tells us that
\begin{equation}\label{eqn.gf}
\|1_A \ast \mu_{B_0'} - \alpha 1_Z\|_{L_2(\mu_Z)}(\|1_A \ast \mu_{B_1'} - \alpha 1_Z\|_{L_2(\mu_Z)} + \alpha)>\eta\alpha^2 - O(\epsilon).
\end{equation}
If
\begin{equation}\label{eqn.gj}
\|1_A \ast \mu_{B_1'} - \alpha 1_Z\|_{L_2(\mu_Z)} + \alpha\geq 2\alpha,
\end{equation}
then
\begin{equation*}
\|1_A \ast \mu_{B_1'} - \alpha 1_Z\|_{L_2(\mu_Z)}^2 \geq \alpha^2,
\end{equation*}
and we arendone by Lemma \ref{lem.l2right} applied to the $1$-step $\epsilon$-closed system
\begin{equation*}
(hB_{0+}h^{-1},hB_0h^{-1},hB_{0-}h^{-1};B_1')
\end{equation*}
and $ A\subset gh^{-1}(hB_0h^{-1})$.  Thus we may suppose that (\ref{eqn.gj}) does not hold and so by (\ref{eqn.gf}) we have
\begin{equation*}
\|1_A \ast \mu_{B_0'} - \alpha 1_Z\|_{L_2(\mu_Z)}>\eta\alpha/2 - O(\epsilon\alpha^{-1}),
\end{equation*}
and so
\begin{equation*}
\|1_A \ast \mu_{B_0'} - \alpha 1_Z\|_{L_2(\mu_Z)}^2=\Omega(\eta^2\alpha^2) - O(\epsilon +\epsilon^2\alpha^{-2}).
\end{equation*}
Now, $\|\mu_{B_1'} \ast \mu_{B_0'} - \mu_{B_0'}\| =O(\epsilon)$ and so
\begin{equation*}
\|1_A \ast \mu_{B_1'} \ast \mu_{B_0'} - \alpha 1_Z\|_{L_2(\mu_Z)}^2 = \|1_A \ast \mu_{B_0'} - \alpha 1_Z\|_{L_2(\mu_Z)}^2 + O(\epsilon \alpha).
\end{equation*}
On the other hand by the convexity of $\|\cdot \|_{L_2(\mu_Z)}$ (and the fact that $B_0' \subset hB_1h^{-1}$ and $Z=gB_0h^{-1}$ so $|ZB_0'| \leq (1+\epsilon)|Z|$) we have
\begin{eqnarray*}
\|1_A \ast \mu_{B_1'} \ast \mu_{B_0'} - \alpha 1_Z\|_{L_2(\mu_Z)}^2 & \leq & \int{\|\rho_{x^{-1}}(1_A \ast \mu_{B_1'}) - \alpha 1_Z\|_{L_2(\mu_Z)}^2d\mu_{B_0'}(x)}\\ & = &  \int{\|\rho_{x^{-1}}(1_A \ast \mu_{B_1'} - \alpha 1_Z)\|_{L_2(\mu_Z)}^2d\mu_{B_0'}(x)} + O(\epsilon\alpha)\\ & = &  \int{\|1_A \ast \mu_{B_1'} - \alpha 1_Z\|_{L_2(\mu_Z)}^2d\mu_{B_0'}(x)} + O(\epsilon)\\ & = & \|1_A \ast \mu_{B_1'} - \alpha 1_Z\|_{L_2(\mu_Z)}^2 + O(\epsilon),
\end{eqnarray*}
from which it follows that
\begin{equation*}
 \|1_A \ast \mu_{B_0'} - \alpha 1_Z\|_{L_2(\mu_Z)}^2 \leq  \|1_A \ast \mu_{B_1'} - \alpha 1_Z\|_{L_2(\mu_Z)}^2 + O(\epsilon).
\end{equation*}
But our lower bound on the left hand side then tells us that
\begin{equation*}
\|1_A \ast \mu_{B_1'} - \alpha 1_Z\|_{L_2(\mu_Z)}^2 = \Omega(\eta^2\alpha^2) - O(\epsilon +\epsilon^2\alpha^{-2}).
\end{equation*}
The result follows again on application of Lemma \ref{lem.l2right}.
\end{proof}
We also have the left analogue of the above.
\begin{lemma}\label{lem.equaliseleft}
Suppose $\mathcal{B}$ and $\mathcal{B}'$ are $1$-step $\epsilon$-closed multiplicative systems with $B_0' \subset gB_1g^{-1}$; and $A \subset Z:=gB_0h^{-1}$ has density $\alpha:=\mu_{Z}(A)>0$; and $\eta \in (0,1]$ is a parameter.  Then
\begin{enumerate}
\item either we have
\begin{equation*}
\|\mu_{B_0'}\ast1_A -\alpha1_{Z}\|_{L_1(\mu_A)} \leq \eta\alpha;
\end{equation*}
\item or there is some $z \in Z$ such that
\begin{equation*}
\mu_{B_1'}\ast 1_A(z) \geq \alpha(1+\Omega(\eta^2))-O(\epsilon\alpha^{-1}).
\end{equation*}
\end{enumerate}
\end{lemma}
The previous two lemmas will provide us with a left and right translate of some suitable multiplicative system on which $A$ has the `right' density.  We now turn to ensuring that the squares of the elements in $A$ have not-too-small density.  This is the lemma for which we need $A$ to have distinct squares.
\begin{lemma}\label{lem.squares}
Suppose that $\mathcal{B}$ is a $1$-step, $\epsilon$-closed multiplicative pair, $S \subset gB_0h^{-1}$ has distinct squares, and $X \subset (gB_1g^{-1}) \cap (hB_1h^{-1})$ is a symmetric neighbourhood of the identity.  Then there is some $s' \in S$ such that
\begin{equation*}
|\{s^2:s \in S\}\cap s'X^4s'| \geq \frac{|X|^2}{(1+\epsilon)^2|B_0|^2} |S|.
\end{equation*}
\end{lemma}
\begin{proof}
We consider the sum
\begin{eqnarray*}
\sum_{s,s' \in S}{|sX \cap s'X||Xs \cap Xs'|}& =& \sum_{z,z' \in G}{\sum_{s,s' \in S}{1_X(s^{-1}z)1_X((s')^{-1}z)1_X(z's^{-1})1_X(z'(s')^{-1})}}\\ & = & \sum_{z \in SX,z' \in XS}{\sum_{s,s' \in S}{1_X(s^{-1}z)1_X((s')^{-1}z)1_X(z's^{-1})1_X(z'(s')^{-1})}}\\ & = & \sum_{z \in SX,z' \in XS}{\left(\sum_{s \in S}{1_X(s^{-1}z)1_X(z's^{-1})}\right)^2}\\ & \geq & \frac{1}{|SX||XS|}\left(\sum_{z \in SX,z' \in XS}{\sum_{s \in S}{1_X(s^{-1}z)1_X(z's^{-1})}}\right)^2\\ & = & \frac{|X|^4|S|^2}{|SX||XS|}.
\end{eqnarray*}
The summands on the left hand side are at most $|X|^2$, and so by averaging there is some $s' \in S$ such that for at least $\frac{|X|^2}{|SX||XS|}|S|$ elements $s \in S$ we have
\begin{equation*}
sX\cap s'X \neq \emptyset \text{ and } Xs \cap Xs' \neq \emptyset.
\end{equation*}
It follows that for each such $s$ we have $s \in s'XX^{-1}$ and $s \in X^{-1}Xs'$ and hence $s^2 \in s'XX^{-2}Xs'=s'X^4s'$. It remains to note that since $X \subset (gB_1g^{-1}) \cap (hB_1h^{-1})$ and $S \subset gB_0h^{-1}$ we have that $XS \subset gB_1B_0h^{-1}\subset gB_{0+}h^{-1}$ and $SX \subset gB_0B_1h^{-1}\subset gB_{0+}h^{-1}$ from which we get the result.
\end{proof}

\begin{proof}[Proof of Proposition \ref{prop.u1}]
Apply Lemmas \ref{lem.equaliseright} and \ref{lem.equaliseleft} with parameter $\eta/4$ and the set $B_0$, to see that either we are in the second or third cases of the proposition (and we are done) or else
\begin{equation*}
\|1_A \ast \mu_{B_0'}-\alpha1_{gB_0h^{-1}}\|_{L_1(\mu_A)} \leq \eta\alpha/4 \text{ and }\|\mu_{B_0'} \ast 1_A -\alpha1_{gB_0h^{-1}}\|_{L_1(\mu_A)} \leq \eta\alpha/4.
\end{equation*}
Let
\begin{equation*}
S:=\{a \in A: |1_A \ast \mu_{B_0'}(a)-\alpha| \leq \eta \alpha \text{ and } |\mu_{B_0'} \ast 1_A(a) -\alpha| \leq \eta \alpha\},
\end{equation*}
so that
\begin{equation*}
\mu_A(A \setminus S)\eta \alpha \leq \|1_A \ast \mu_{B_0'}-\alpha1_{gB_0h^{-1}}\|_{L_1(\mu_A)} + \|\mu_{B_0'} \ast 1_A -\alpha1_{gB_0h^{-1}}\|_{L_1(\mu_A)} \leq \eta \alpha/2.
\end{equation*}
It follows that $\mu_A(S) \geq 1/2$.  Now apply Lemma \ref{lem.squares} to the set $S \subset gB_0h^{-1}$ and
\begin{equation*}
X \subset X^4 \subset B_1' \subset B_0' \subset gB_1g^{-1} \cap hB_1h^{-1}
\end{equation*}
to get that there is some $a \in S$ such that
\begin{equation*}
|\{s^2:s \in S\}\cap aX^4a| \geq \frac{|X|^2}{(1+\epsilon)^2|B_0|^2}|S|.
\end{equation*}
Since $a \in S$ we have that
\begin{equation*}
1_A \ast \mu_{B_0'}(a) \geq \alpha(1-\eta) \text{ and } \mu_{B_0'} \ast 1_A(a) \geq \alpha(1-\eta) ,
\end{equation*}
and since $S\subset A$ and $X^4 \subset B_1'$ we have
\begin{equation*}
\mu_{aB_1'a}(\{s^2:s \in A\}) =\Omega(\alpha \delta^2),
\end{equation*}
and the result is proved.
\end{proof}

We now turn to the companion result to Proposition \ref{prop.u1} which explains how to use the output in the first case of that result.
\begin{proposition}\label{prop.u2}
Suppose that $\mathcal{B}$ is a $2$-step $\epsilon$-closed multiplicative system; that there is a symmetric neighbourhood of the identity $X$ of density $\delta:=\mu_G(X)>0$ such that $X^4 \subset B_2$; and that $U,V \subset B_{0-}$ have $\mu_{B_0}(U)=\mu_{B_0}(V)=\alpha>0$ and $W \subset B_{1-}$ has $\mu_{B_1}(W)=\omega>0$. Then
\begin{enumerate}
\item either
\begin{equation*}
\langle 1_{U} \ast \mu_V,1_{W}\rangle_{L_2(\mu_{B_1})} \geq  \frac{1}{2}\mu_{B_0}(U)\mu_{B_0}(V)\mu_{B_1}(W);
\end{equation*}
\item or there is a $1$-step $\epsilon$-closed multiplicative system $\mathcal{B}^L$ and a symmetric neighbourhood of the identity $S_L$ such that $S_L^4 \subset B^L_1$ and
\begin{equation*}
\mu_G(S_L) \geq \exp(-O(\epsilon^{-1}\alpha^{-1}\log 2\omega^{-1}\log 2\delta^{-1})^{O(1)})
\end{equation*}
and some $z \in B_0$ with
\begin{equation*}
\mu_{B_0^Lz}(U) \geq \mu_{B_0}(U)(1+\Omega(1)) - O(\epsilon \alpha^{-1});
\end{equation*}
\item or there is a $1$-step $\epsilon$-closed multiplicative system $\mathcal{B}^R$ and a symmetric neighbourhood of the identity $S_R$ such that $S_R^4 \subset B^R_1$ and
\begin{equation*}
\mu_G(S_R) \geq \exp(-O(\epsilon^{-1}\alpha^{-1}\log 2\omega^{-1}\log 2\delta^{-1})^{O(1)})
\end{equation*}
and some $z \in B_0$ with
\begin{equation*}
\mu_{zB_0^R}(V) \geq \mu_{B_0}(V)(1+\Omega(1)) - O(\epsilon \alpha^{-1}).
\end{equation*}
\end{enumerate}
\end{proposition}
\begin{proof}
Begin by applying Lemma \ref{lem.nab} to get a symmetric neighbourhood of the identity $T$ such that $T^8 \subset X^4$ and $\mu_G(T) \geq \exp(-O(\log^22\delta^{-1}))$.

Let $\eta:=\epsilon\alpha \in (0,1]$ and $p \in [2,\infty)$ be a parameter to be chosen later (it will be apparent that it only depends on data available at this stage of the proof) and apply Lemma \ref{lem.csrel} to get that there is a symmetric neighbourhood of the identity $R \subset T^2$ with
\begin{equation*}
\mu_G(R) \geq \exp(-O(\eta^{-2}p\log 2\mu_G(V)^{-1} + \log^22\delta^{-1}))
\end{equation*}
such that for all $t \in R^4$ we have
\begin{equation*}
\|\rho_{t^{-1}}(1_{U} \ast \mu_{V}) - 1_U \ast \mu_{V}\|_{L_p(\mu_{B_1})} \leq \eta(1+O(\epsilon/p))\|1_{U}\|_{L_\infty(\mu_{B_{0+}})} =O(\eta).
\end{equation*}
Apply Corollary \ref{cor.cont} to get a $1$-step $\epsilon$-closed multiplicative system $\mathcal{B}^R$ and symmetric neighbourhood of the identity $S_R$ such that
\begin{equation*}
B^R_0 \subset R^4\text{ and }S_R^4 \subset B^R_1
\end{equation*}
and
\begin{equation*}
\mu_G(S_R) \geq \exp(-O(\epsilon^{-1}\log2\mu_G(R)^{-1})^{O(1)}).
\end{equation*}
By integrating we have
\begin{equation*}
\|1_{U} \ast \mu_{V}\ast \mu_{B^R_0} - 1_U \ast \mu_{V}\|_{L_p(\mu_{B_1})} =O(\eta).
\end{equation*}
It follows from linearity and H{\"o}lder's inequality (writing $p'$ for the conjugate index of $p$) that
\begin{equation*}
|\langle1_{U} \ast \mu_{V} \ast \mu_{B^R_0},1_{W}\rangle_{L_2(\mu_{B_1})}-\langle 1_U \ast \mu_V,1_W\rangle_{L_2(\mu_{B_1})}| = O(\eta\|1_{W}\|_{L_{p'}(\mu_{B_1})}).
\end{equation*}
Now, $\|1_W\|_{L_{p'}(\mu_{B_1})} = \mu_{B_1}(W)^{1/p'}=\mu_{B_1}(W)^{1-1/p}$ and so taking $p = 2+\log\mu_{B_1}(W)^{-1}$ we see that
\begin{equation*}
\langle 1_U \ast \mu_V,1_W\rangle_{L_2(\mu_{B_1})}=\langle 1_{U} \ast \mu_{V} \ast \mu_{B^R_0},1_{W}\rangle_{L_2(\mu_{B_1})} + O(\eta\mu_{B_1}(W)).
\end{equation*}
Since $U \subset B_{0-}$ we see that
\begin{equation*}
1_U \ast \mu_{B_0}(x) = \mu_{B_0}(U) \text{ for all } x \in B_1
\end{equation*}
and so
\begin{eqnarray}\label{eqn.earlier}
\langle 1_U \ast \mu_V,1_W\rangle_{L_2(\mu_{B_1})} & = & \langle 1_U \ast (\mu_V\ast \mu_{B^R_0} - \mu_{B_0}),1_W \rangle_{L_2(\mu_{B_1})}\\ \nonumber & & + \mu_{B_0}(U)\langle 1_{B_1},1_W \rangle_{L_2(\mu_{B_1})} + O(\eta\mu_{B_1}(W)).
\end{eqnarray}
On the other hand, if we write $f:=1_V\ast \mu_{B^R_0} - \mu_{B_0}(V)1_{B_0}$ then
\begin{equation*}
\langle 1_U \ast (\mu_V \ast \mu_{B^R_0} -\mu_{B_0}(V) \mu_{B_0}),1_W\rangle_{L_2(\mu_{B_1})} = \langle \mu_U \ast f, 1_W\rangle_{L_2(\mu_{B_1})}.
\end{equation*}
Now we can apply Lemma \ref{lem.csrelleft} (with the same parameters as before) to get that there is a symmetric neighbourhood of the identity $L \subset T^2$ with
\begin{equation*}
\mu_G(L) \geq \exp(-O(\eta^{-2}p\log 2\mu_G(U)^{-1} + \log^22\delta^{-1}))
\end{equation*}
such that for all $t \in L^4$ we have
\begin{equation*}
\|\lambda_t( \mu_{U}\ast f)- \mu_{U} \ast f\|_{L_p(\mu_{B_1})} =O(\eta \|f\|_{L_\infty(\mu_{B_{0+}})})=O(\eta),
\end{equation*}
and argue as before.  Apply Corollary \ref{cor.cont} to get a $1$-step $\epsilon$-closed multiplicative system $\mathcal{B}^L$ and symmetric neighbourhood of the identity $S_L$ such that
\begin{equation*}
B^L_0 \subset L^4\text{ and }S_L^4 \subset B^L_1
\end{equation*}
and
\begin{equation*}
\mu_G(S_L) \geq \exp(-O(\epsilon^{-1}\log2\mu_G(L)^{-1})^{O(1)}).
\end{equation*}
By integrating we have
\begin{equation*}
\|\mu_{B^L_0} \ast\mu_U \ast f - \mu_U\ast f\|_{L_p(\mu_{B_1})} =O(\eta).
\end{equation*}
Our choice of $p$ ensures that
\begin{equation*}
\langle \mu_U \ast f,1_W\rangle_{L_2(\mu_{B_1})}=\langle \mu_{B_0^L} \ast \mu_{U} \ast f,1_{W}\rangle_{L_2(\mu_{B_1})} + O(\eta\mu_{B_1}(W)).
\end{equation*}
Since $V \subset B_0^-$ we see that for all $x \in B_1$ we have
\begin{align*}
\mu_{B_0} \ast f(x) & \leq \mu_{B_0} \ast 1_V \ast \mu_{B_0^R}(x) - \mu_{B_0}(V) \mu_{B_0} \ast 1_{B_0}(x)\\
& \leq \mu_{B_0}(V) - (1-O(\epsilon))\mu_{B_0}(V)=O(\epsilon\mu_{B_0}(V)),
\end{align*}
and so
\begin{eqnarray}
\label{eqn.o} \langle \mu_U \ast f,1_W\rangle_{L_2(\mu_{B_1})} & = & \langle (\mu_{B^L_0}\ast \mu_U-\mu_{B_0}) \ast f,1_W \rangle_{L_2(\mu_{B_1})}\\ \nonumber & & + O(\epsilon\mu_{B_0}(V))\mu_{B_1}(W) + O(\eta\mu_{B_1}(W)).
\end{eqnarray}
Put
\begin{equation*}
g:=\mu_{B_0^L}\ast 1_U-\mu_{B_0}(U)1_{B_0} \text{ and } \nu :=\mu_{B_0}(U)(\mu_{B_0^L}\ast \mu_U-\mu_{B_0})
\end{equation*}
so that
\begin{equation*}
\nu \ast h(x)= \langle \rho_x(h),\tilde{g}\rangle_{L_2(\mu_{B_0})} \text{ for all }h:G \rightarrow \C.
\end{equation*}
With these definitions, (\ref{eqn.o}) and (\ref{eqn.earlier}) give
\begin{align*}
\langle \nu\ast f,1_W\rangle_{L_2(\mu_{B_1})} & = \mu_{B_0}(U)\langle \mu_U \ast f,1_W\rangle_{L_2(\mu_{B_1})} - O((\epsilon\mu_{B_0}(V)+\eta)\mu_{B_1}(W)\mu_{B_0}(U))\\
& = \mu_{B_0}(U)\mu_{B_0}(V)\mu_{B_1}(W) -  O((\epsilon\mu_{B_0}(V)+\eta)\mu_{B_1}(W)\mu_{B_0}(U)).
\end{align*}
We conclude that either we are in the first case of the lemma or else there is some $x \in B_1$ such that
\begin{equation*}
|\nu\ast f(x)| \geq \frac{1}{2}\mu_{B_0}(U)\mu_{B_0}(V)(1-O(\epsilon))
\end{equation*}
(since $\eta=\epsilon \alpha$).  Now, by the Cauchy-Schwarz inequality we have
\begin{eqnarray*}
|\nu \ast f(x)| &= & |\langle \rho_x(f),\tilde{g}\rangle_{L_2(\mu_{B_0})}|\\ & \leq & \|\rho_x(f)\|_{L_2(\mu_{B_0})}\|\tilde{g}\|_{L_2(\mu_{B_0})} = (\|f\|_{L_2(\mu_{B_0})}+O(\epsilon))\|g\|_{L_2(\mu_{B_0})}.
\end{eqnarray*}
It follows that either
\begin{equation*}
\|g\|_{L_2(\mu_{B_0})}^2 \geq \frac{1}{2}\mu_{B_0}(U)^2-O(\epsilon)
\end{equation*}
or
\begin{equation}\label{eqn.second}
\|f\|_{L_2(\mu_{B_0})}^2 \geq \frac{1}{2}\mu_{B_0}(V)^2-O(\epsilon).
\end{equation}
In the first instance we apply Lemma \ref{lem.l2left} to the $1$-step $\epsilon$-closed multiplicative system
\begin{equation*}
(B_{0+},B_0,B_{0-};B^L_0)
\end{equation*}
and the set $U \subset B_0$ to get that there is some $z\in B_0$ such that
\begin{equation*}
\mu_{B^L_0z}(U) \geq \mu_{B_0}(U)(1+\Omega(1)) - O(\epsilon\alpha^{-1}).
\end{equation*}
We are in the second case of the proposition with the system $\mathcal{B}^L$. In the second instance above (\ref{eqn.second}) we apply Lemma \ref{lem.l2right} and are in the third case of the proposition.  The result is proved.
\end{proof}

\section{Proof of the main result}

We are now in a position to prove out main theorem.
\begin{proof}[Proof of Theorem \ref{thm.main}]  We fix two parameters $\epsilon=c'\alpha^2$ and $\eta=c$ for some absolute constants $c,c'>0$ whose precise value will become clear later in the argument.  All of the multiplicative systems we consider will be $\epsilon$-closed.  

We shall proceed iteratively.  At stage $i$ of the iteration we suppose that we have the following data:
\begin{enumerate}
\item $\mathcal{B}^{(i)}$ a $1$-step ($\epsilon$-closed) multiplicative system;
\item $X_i$, a symmetric neighbourhood of the identity with density $\delta_i:=\mu_G(X_i)>0$ such that $X_i^4 \subset B_1^{(i)}$;
\item $g_i,h_i \in G$ such that $A_i:=A \cap g_iB_0^{(i)}h_i^{-1}$ has $\alpha_i:=\mu_{g_iB_0^{(i)}h_i^{-1}}(A_i)>0$.
\end{enumerate}
We initialise with the $1$-step $\epsilon$-closed multiplicative system $\mathcal{B}^{(0)}:=(G,G,G;G)$, $X_0:=G$, $g_0,h_0=1_G$ and $A_0:=A$ so that
\begin{equation*}
\alpha_0=\alpha>0 \text{ and } \delta_0=1>0.
\end{equation*}
At some stage $i_0$ the iteration will terminate but for all $0 \leq i < i_0$ we shall have
\begin{equation*}
\alpha_i \geq \alpha(1+\Omega(1))^i \text{ and } \delta_i \geq \exp(-O(\alpha^{-1})^{\exp(O(i))}).
\end{equation*}
At stage $i$ of our iteration apply Lemma \ref{lem.int} to the set $X_i$ and elements $g_i$ and $h_i$ to get a symmetric neighbourhood of the identity $Y_i$ such that
\begin{equation*}
Y_i^4 \subset (g_iX_i^4g_i^{-1}) \cap (h_iX_i^4h_i^{-1}) \text{ and } \mu_G(Y_i) \geq \exp(-O(\log^22\delta_i^{-1})).
\end{equation*}
Now apply Corollary \ref{cor.cont} to $Y_i$ to get a $2$-step multiplicative system $\mathcal{B}^{(i)'}$ and a symmetric neighbourhood of the identity $S_i$ such that
\begin{equation*}
B^{(i)'}_0 \subset Y_i^4, S_i^4 \subset B^{(i)'}_3 \text{ and } \mu_G(S_i)\geq \exp(-O(\epsilon^{-1}\log 2\delta_i^{-1})^{O(1)}).
\end{equation*}
Apply Proposition \ref{prop.u1} to $\mathcal{B}^{(i)}$ and the $1$-step $\epsilon$-closed system
\begin{equation*}
(B^{(i)'}_{0+},B^{(i)'}_0,B^{(i)'}_{0-};B^{(i)'}_{1-}),
\end{equation*}
which have
\begin{equation*}
B^{(i)'}_0 \subset Y_i^4 \subset (g_iX_i^4g_i^{-1}) \cap (h_iX_i^4h_i^{-1})  \subset g_iB^{(i)}_1g_i^{-1} \cap h_iB^{(i)}_1h_i^{-1} 
\end{equation*}
by design, and $A_i \subset g_iB_0^{(i)}h_i^{-1}$ and $S_i^4 \subset B^{(i)'}_2\subset B^{(i)'}_{1-}$.

If we are in the second two cases of Proposition \ref{prop.u1} then we have some $z_i \in g_iB_0^{(i)}h_i^{-1}$ such that either
\begin{align}\label{eqn.caas1}
\mu_{B_{1}^{(i)'}z_i}(A_i) & \geq \mu_{B_{1-}^{(i)'}z_i}(A_i)(1-O(\epsilon))\\ \nonumber & = \mu_{B_{1-}^{(i)'}} \ast 1_{A_i}(z_i)(1-O(\epsilon)) \geq \alpha_i(1+\Omega(\eta^2)) - O(\epsilon\alpha_i^{-1})
\end{align}
or
\begin{align}\label{eqn.caas2}
\mu_{z_iB_{1}^{(i)'}}(A_i) &\geq \mu_{z_iB_{1-}^{(i)'}}(A_i)(1-O(\epsilon))\\ \nonumber & = 1_{A_i}\ast\mu_{B_{1-}^{(i)'}}(z_i)(1-O(\epsilon)) \geq \alpha_i(1+\Omega(\eta^2)) - O(\epsilon\alpha_i^{-1}).
\end{align}
Set
\begin{equation*}
\mathcal{B}^{(i+1)}:=(B_{1+}^{(i)'},B_1^{(i)'},B_{1-}^{(i)'};B_2^{(i)'})
\end{equation*}
and $X_{i+1}:=S_i$ so that we have
\begin{equation*}
X_{i+1}^4 = S_i^4 \subset B_2^{(i)'}=B^{(i+1)}_1 \text{ and } \delta_{i+1} \geq \exp(-O(\epsilon^{-1}\log 2\delta_i^{-1})^{O(1)}).
\end{equation*}
It follows by the inductive hypothesis and the value of $\epsilon$ that
\begin{equation*}
\delta_{i+1} = \exp(-O(\alpha^{-1})^{\exp(O((i+1)))}).
\end{equation*}
If (\ref{eqn.caas1}) holds then take $g_{i+1}=1_G$, $h_{i+1}=z_i^{-1}$, so
\begin{align*}
\alpha_{i+1} & =\mu_{g_{i+1}B_0^{(i+1)}h_{i+1}^{-1}}(A \cap g_{i+1}B_0^{(i+1)}h_{i+1}^{-1})\\ & = \mu_{B_1^{(i)'}z_i}(A) \geq \mu_{B_1^{(i)'}z_i}(A_i)\\ &\geq \alpha_i(1+\Omega(\eta^2)) - O(\epsilon\alpha_i^{-1}) = \alpha_i(1+\Omega(c^2)-O(c'))=\alpha_i(1+\Omega(1))
\end{align*}
provided $c'$ is sufficiently small compared with $c^2$.  On the other hand, if (\ref{eqn.caas2}) holds then take  $g_{i+1}=z_i$ and $h_{i+1}=1_G$, so
\begin{align*}
\alpha_{i+1} & =\mu_{g_{i+1}B_0^{(i+1)}h_{i+1}^{-1}}(A \cap g_{i+1}B_0^{(i+1)}h_{i+1}^{-1})\\ & = \mu_{z_iB_1^{(i)'}}(A) \geq \mu_{z_iB_1^{(i)'}}(A_i)\\&\geq \alpha_i(1+\Omega(\eta^2)) - O(\epsilon\alpha_i^{-1}) = \alpha_i(1+\Omega(c^2)-O(c'))=\alpha_i(1+\Omega(1)),
\end{align*}
again provided $c'$ is sufficiently small compared with $c^2$.  In either case
\begin{equation*}
\alpha_{i+1} \geq \alpha_i(1+\Omega(1)) \geq \alpha(1+\Omega(1))^{i+1}.
\end{equation*}
With the easier cases dealt with, suppose that we are in the first case of Proposition \ref{prop.u1} and there is some $a_i$ such that
\begin{equation*}
\mu_{a_iB_0^{(i)'}}(A_i),\mu_{B_0^{(i)'}a_i}(A_i) \geq \alpha_i(1-\eta)
\end{equation*}
and
\begin{equation*}
\mu_{a_iB_{1-}^{(i)'}a_i}(\{s^2:s \in A_i\}) =\Omega(\alpha_i\mu_G(S_i)^2).
\end{equation*}
Now let
\begin{equation*}
\tilde{U}_i:=(a_i^{-1}A_i)\cap B_{0-}^{(i)'} \text{ and } \tilde{V}_i:=(A_ia_i^{-1})\cap B_{0-}^{(i)'},
\end{equation*}
and note that
\begin{equation*}
\min\left\{\mu_{B_{0}^{(i)'}}(\tilde{U}_i) ,\mu_{B_{0}^{(i)'}}(\tilde{V}_i) \right\} \geq \alpha_i(1-\eta) -O(\epsilon).
\end{equation*}
Thus we can pick $U_i \subset \tilde{U}_i$ and $V_i \subset \tilde{V}_i$ such that
\begin{equation*}
\mu_{B_{0}^{(i)'}}(U_i)=\mu_{B_{0}^{(i)'}}(V_i) =\min\left\{\mu_{B_{0}^{(i)'}}(\tilde{U}_i) ,\mu_{B_{0}^{(i)'}}(\tilde{V}_i) \right\} \geq \alpha_i(1-\eta) -O(\epsilon).
\end{equation*}
Let
\begin{equation*}
W_i:=\{a_i^{-1}s^2a_i^{-1}: s \in A_i\}\cap B_{1-}^{(i)'}.
\end{equation*}
so that
\begin{equation*}
\mu_{B_{1-}^{(i)'}}(W_i) =\Omega(\alpha_i \mu_G(S_i)^2) = \Omega(\alpha_i  \exp(-O(\epsilon^{-1}\log 2\delta_i^{-1})^{O(1)})).
\end{equation*}
We apply Proposition \ref{prop.u2} to the system $\mathcal{B}^{(i)'}$, the symmetric neighbourhood of the identity $S_i$ which has $S_i^4 \subset B_2^{(i)'}$, the sets $U_i,V_i \subset B_{0-}^{(i)'}$ which have
\begin{equation*}
\mu_{B_{0}^{(i)'}}(U_i)=\mu_{B_{0}^{(i)'}}(V_i) \geq \alpha_i(1-\eta) -O(\epsilon),
\end{equation*}
and the set $W_i \subset B_{1-}^{(i)'}$.  If we are in the first case we terminate our iteration with
\begin{equation}\label{eqn.termination}
\langle 1_{U_i} \ast \mu_{V_i},1_{W_i}\rangle_{L_2\left(\mu_{B_1^{(i)'}}\right)} \geq \frac{1}{2}\mu_{B_0^{(i)'}}(U_i)\mu_{B_0^{(i)'}}(V_i)\mu_{B_1^{(i)'}}(W_i).
\end{equation}
If we are in the second case of Proposition \ref{prop.u2} then let $\mathcal{B}^{(i+1)}$ be the multiplicative system, and $X_{i+1}$ the given symmetric neighbourhood of the identity so that
\begin{align*}
\delta_{i+1} &\geq \exp(-O(\epsilon^{-1}(\alpha_i(1-\eta) -O(\epsilon))^{-1}\log 2\mu_{B_1}(W)^{-1} \log 2\delta_i^{-1})^{O(1)})\\ &=\exp(-O(\epsilon^{-1}\alpha^{-1}\log2\delta_i)^{O(1)})= \exp(-O(\alpha^{-1})^{\exp(O((i+1)))}).
\end{align*}
We are given an $h_i$ such that
\begin{equation*}
\mu_{B^{(i+1)}_0h_i^{-1}}(U_i) \geq \mu_{B_0^{(i)'}}(U_i)(1+\Omega(1))-O(\epsilon\alpha_i^{-1}).
\end{equation*}
It follows that putting $g_i:=a_i$ we have
\begin{eqnarray*}
\alpha_{i+1} \geq \mu_{g_iB^{(i+1)}_0h_i^{-1}}(A_i \cap g_iB^{(i+1)}_0h_i^{-1}) & \geq & \mu_{B_0^{(i)'}}(U)(1+\Omega(1))-O(\epsilon\alpha_i^{-1})\\ & \geq & \alpha_i(1+\Omega(1) - \eta) - O(\epsilon \alpha^{-1})\\
& = & \alpha_i(1+\Omega(1)-c-O(c'))=\alpha_i(1+\Omega(1))
\end{eqnarray*}
provided $c$ and $c'$ are sufficiently small absolute constants.  If we are in the third case of Proposition \ref{prop.u2} then let $\mathcal{B}^{(i+1)}$ be the multiplicative system, and $X_{i+1}$ the given symmetric neighbourhood of the identity so that
\begin{align*}
\delta_{i+1} &\geq \exp(-O(\epsilon^{-1}(\alpha_i(1-\eta) -O(\epsilon))^{-1}\log 2\mu_{B_1}(W)^{-1} \log 2\delta_i^{-1})^{O(1)})\\ &=\exp(-O(\epsilon^{-1}\alpha^{-1}\log2\delta_i)^{O(1)})= \exp(-O(\alpha^{-1})^{\exp(O((i+1)))}).
\end{align*}
We are given an $g_i$ such that
\begin{equation*}
\mu_{g_iB^{(i+1)}_0}(V_i) \geq \mu_{B_0^{(i)'}}(V_i)(1+\Omega(1))-O(\epsilon\alpha_i^{-1}).
\end{equation*}
It follows that putting $h_i:=a_i^{-1}$ we have
\begin{eqnarray*}
\alpha_{i+1} \geq \mu_{g_iB^{(i+1)}_0h_i^{-1}}(A_i \cap g_iB^{(i+1)}_0h_i^{-1}) & \geq & \mu_{B_0^{(i)'}}(V)(1+\Omega(1))-O(\epsilon\alpha_i^{-1})\\ & \geq & \alpha_i(1+\Omega(1) - \eta) - O(\epsilon \alpha^{-1})\\
& = & \alpha_i(1+\Omega(1)-c-O(c'))=\alpha_i(1+\Omega(1)),
\end{eqnarray*}
again provided $c$ and $c'$ are sufficiently small absolute constants.  In either of the above cases
\begin{equation*}
\alpha_{i+1}\geq(1+\Omega(1))\alpha_i \geq \alpha(1+\Omega(1))^{i+1}.
\end{equation*}
Since $\alpha_i \leq 1$ for all $i$ it follows that the iteration terminates at some step $i_0$ with
\begin{equation*}
1 \geq \alpha(1+\Omega(1))^{i_0},
\end{equation*}
\emph{i.e.} at some stage $i_0=O(\log 2\alpha^{-1})$.  It follows from this that
\begin{equation*}
\delta_{i_0} \geq \exp(-\exp(\alpha^{-O(1)})).
\end{equation*}
When we terminate the iteration (\ref{eqn.termination}) holds and so there are at least
\begin{align*}
|B_0^{(i)'}||B_1^{(i)'}|\cdot  \frac{1}{2}\mu_{B_0^{(i)'}}(U_i)\mu_{B_0^{(i)'}}(V_i)\mu_{B_1^{(i)'}}(W_i) & = \Omega(\alpha_i^3\mu_G(S_i)^4)|G|^2\\ & = \Omega(\alpha^3\exp(-O(\alpha^{-1}\log 2\delta_{i_0}^{-1})^{O(1)}))|G|^2
\end{align*}
triples $(r,s,t) \in U_i \times W_i\times V_i$ such that $rt=s$.  The mapping taking $(r,s,t)$ to $(a_ir,a_isa_i,ta_i)$ is an injection into
\begin{equation*}
a_iU_i \times a_iW_ia_i \times V_ia_i \subset A \times \{a^2:a \in A\} \times A,
\end{equation*}
and $(a_ir)(ta_i)=a_i(rt)a_i=a_isa_i$.  It follows that there are at least
\begin{equation*}
\Omega(\alpha^3\exp(-O(\alpha^{-1}\log 2\delta_{i_0}^{-1})^{O(1)}))|G|^2 = \exp(-\exp(\alpha^{-O(1)}))|G|^2
\end{equation*}
triples in $(x,y,z) \in A \times A \times A$ with $xz=y^2$ as claimed.
\end{proof}

\section*{Acknowledgement}

We should like to thank the referee for a careful reading of the paper which identified numerous errors.

\bibliographystyle{halpha}

\bibliography{references}

\begin{thebibliography}{BMZ97}

\bibitem[BB82]{brobuh::0}
T.~C. Brown and J.~P. Buhler.
\newblock A density version of a geometric {R}amsey theorem.
\newblock {\em J. Combin. Theory Ser. A}, 32(1):20--34, 1982.

\bibitem[BGT10]{bregretao::}
E.~{Breuillard}, B.~{Green}, and T.~{Tao}.
\newblock {Approximate subgroups of linear groups}.
\newblock {\em ArXiv e-prints}, May 2010, 1005.1881.

\bibitem[Blo14]{blo::0}
T.~F. Bloom.
\newblock A quantitative improvement for {R}oth's theorem on arithmetic
  progressions.
\newblock 2014, arXiv:1405.5800.

\bibitem[BMZ97]{bermcczha::}
V.~Bergelson, R.~McCutcheon, and Q.~Zhang.
\newblock A {R}oth theorem for amenable groups.
\newblock {\em Amer. J. Math.}, 119(6):1173--1211, 1997.

\bibitem[Bog39]{bog::}
N.~Bogolio{\`u}boff.
\newblock Sur quelques propri\'et\'es arithm\'etiques des presque-p\'eriodes.
\newblock {\em Ann. Chaire Phys. Math. Kiev}, 4:185--205, 1939.

\bibitem[Bou99]{bou::5}
J.~Bourgain.
\newblock On triples in arithmetic progression.
\newblock {\em Geom. Funct. Anal.}, 9(5):968--984, 1999.

\bibitem[CF12]{confox::}
D~Conlon and J.~Fox.
\newblock Graph removal lemmas.
\newblock 2012, arXiv:1211.3487.

\bibitem[Cha02]{cha::0}
M.-C. Chang.
\newblock A polynomial bound in {F}re{\u\i}man's theorem.
\newblock {\em Duke Math. J.}, 113(3):399--419, 2002.

\bibitem[CS10]{crosis::}
E.~S. Croot and O.~Sisask.
\newblock A probabilistic technique for finding almost-periods of convolutions.
\newblock {\em Geom. Funct. Anal.}, 20(6):1367--1396, 2010.

\bibitem[FGR87]{fragrarod::}
P.~Frankl, R.~L. Graham, and V.~R{\"o}dl.
\newblock On subsets of abelian groups with no {$3$}-term arithmetic
  progression.
\newblock {\em J. Combin. Theory Ser. A}, 45(1):157--161, 1987.

\bibitem[Fox11]{fox::}
J.~Fox.
\newblock A new proof of the graph removal lemma.
\newblock {\em Ann. of Math. (2)}, 174(1):561--579, 2011, arXiv:1006.1300.

\bibitem[Gow08]{gow::2}
W.~T. Gowers.
\newblock Quasirandom groups.
\newblock {\em Comb. Probab. Comput.}, 17(3):363--387, 2008.

\bibitem[GW11]{gowwol::}
W.~T. Gowers and J.~Wolf.
\newblock Linear forms and quadratic uniformity for functions on
  {$\mathbb{Z}_N$}.
\newblock {\em J. Anal. Math.}, 115:121--186, 2011, arXiv:1002.2210.

\bibitem[HB87]{hea::}
D.~R. Heath-Brown.
\newblock Integer sets containing no arithmetic progressions.
\newblock {\em J. London Math. Soc. (2)}, 35(3):385--394, 1987.

\bibitem[KSV09]{kraserven::}
D.~Kr{\'a}l, O.~Serra, and L.~Vena.
\newblock A combinatorial proof of the removal lemma for groups.
\newblock {\em J. Combin. Theory Ser. A}, 116(4):971--978, 2009.

\bibitem[Mes95]{mes::}
R.~Meshulam.
\newblock On subsets of finite abelian groups with no {$3$}-term arithmetic
  progressions.
\newblock {\em J. Combin. Theory Ser. A}, 71(1):168--172, 1995.

\bibitem[Pyb97]{pyb::}
L.~Pyber.
\newblock How abelian is a finite group?
\newblock In {\em The mathematics of {P}aul {E}rd{\H o}s, {I}}, volume~13 of
  {\em Algorithms Combin.}, pages 372--384. Springer, Berlin, 1997.

\bibitem[Rot52]{rot::}
K.~F. Roth.
\newblock Sur quelques ensembles d'entiers.
\newblock {\em C. R. Acad. Sci. Paris}, 234:388--390, 1952.

\bibitem[Rot53]{rot::0}
K.~F. Roth.
\newblock On certain sets of integers.
\newblock {\em J. London Math. Soc.}, 28:104--109, 1953.

\bibitem[RS78]{ruzsze::}
I.~Z. Ruzsa and E.~Szemer{\'e}di.
\newblock Triple systems with no six points carrying three triangles.
\newblock In {\em Combinatorics (Proc. Fifth Hungarian Colloq., Keszthely,
  1976), Vol. II}, volume~18 of {\em Colloq. Math. Soc. J\'anos Bolyai}, pages
  939--945. North-Holland, Amsterdam, 1978.

\bibitem[Sol13]{sol::0}
J.~Solymosi.
\newblock Roth-type theorems in finite groups.
\newblock {\em European J. Combin.}, 34(8):1454--1458, 2013.

\bibitem[Sze90]{sze::2}
E.~Szemer{\'e}di.
\newblock Integer sets containing no arithmetic progressions.
\newblock {\em Acta Math. Hungar.}, 56(1-2):155--158, 1990.

\bibitem[TV06]{taovu::}
T.~C. Tao and H.~V. Vu.
\newblock {\em Additive combinatorics}, volume 105 of {\em Cambridge Studies in
  Advanced Mathematics}.
\newblock Cambridge University Press, Cambridge, 2006.

\end{thebibliography}

\end{document}